\documentclass[12pt]{article}
\usepackage{amsmath,amsfonts,amssymb,amsthm}
\usepackage{mathtools}
\usepackage{tikz}
\usepackage{enumitem,nicematrix}
\usetikzlibrary{arrows,arrows.meta}
\DeclareMathOperator{\cl}{cl}
\newcommand{\prob}[3]{\begin{trivlist}\item[]
    #1\\
    \textbf{Input:} #2\\
    \textbf{Output:} #3
\end{trivlist}}
\newcommand{\comps}{\mathbb{C}}
\title{The Complexity of the Greedoid Tutte Polynomial}
\author{Christopher Knapp\\
\small Brunel University\\[-0.8ex]
\small Kingston Lane\\[-0.8ex]
\small Uxbridge, UK\\
\small\tt christopher.knapp2@brunel.ac.uk
\and
Steven Noble\\
\small Birkbeck, University of London\\[-0.8ex]
\small Malet Street\\[-0.8ex]
\small London, UK\\
\small\tt steven.noble@bbk.ac.uk}

\newtheorem{theorem}{Theorem}
\newtheorem{proposition}[theorem]{Proposition}
\newtheorem{lemma}[theorem]{Lemma}
\newtheorem{corollary}[theorem]{Corollary}
\theoremstyle{definition}
\newtheorem{definition}[theorem]{Definition}
\newtheorem{example}[theorem]{Example}
\begin{document}

\maketitle

\begin{abstract}
We consider the Tutte polynomial of three classes of greedoids: those arising from rooted graphs, rooted digraphs and binary matrices. We establish the computational complexity of evaluating each of these polynomials at each fixed rational point $(x,y)$. In each case we show that evaluation is $\#$P-hard except for a small number of exceptional cases when there is a polynomial time algorithm. In the binary case, establishing $\#$P-hardness along one line relies on Vertigan's unpublished result on the complexity of counting bases of a matroid. For completeness, we include an appendix providing a proof of this result.


\end{abstract}

\section{Introduction}

Tutte's eponymous polynomial is perhaps the most widely studied two-variable graph and matroid polynomial due to its many specializations, their
vast breadth and the richness of the underlying theory. Discussion of the Tutte polynomial and closely related polynomials fills an entire handbook~\cite{zbMATH07553843}.
Tutte first introduced the Tutte polynomial of a graph, as the \emph{dichromate} in~\cite{zbMATH03087501}. It is closely related to Whitney's rank
generating function~\cite{MR1503085} which Tutte extended from graphs to matroids in his PhD thesis~\cite{Tutte-thesis}.
Crapo~\cite{MR262095} later extended the definition of the Tutte polynomial to matroids.
See Farr~\cite{Farr-Chapter} for more on the early history of the Tutte polynomial.

The simplest definition of the Tutte polynomial $T(G;x,y)$ of a graph $G$ is probably in terms of the rank function $r$. Given a graph $G$ and set $A$ of its edges, we have $r(A)=|V(G)|-k(G|A)$, where $k(G|A)$ is the number of connected components of the graph obtained from $A$ by deleting the edges in $E(G)-A$ (and keeping all the vertices).
\begin{definition}
For a graph $G$ with edge set $E$, we have
\[ T(G;x,y) =   \sum_{A\subseteq E} (x-1)^{r(E)-r(A)} (y-1)^{|A|-r(A)}.\]
\end{definition}

By making appropriate substitutions for $x$ and $y$, a huge number of graph invariants with connections to diverse areas of mathematics may be obtained.
We summarise just a few of these evaluations that are particularly relevant later in this paper. A \emph{spanning subgraph} of a graph $G$ is a subgraph including all the vertices of $G$.
\begin{itemize}
\item $T(G;1,1)$  is the number of maximal spanning forests of $G$. (If $G$ is connected, then this is the number of spanning trees.)
\item $T(G;2,1)$ is the number of spanning forests of $G$.
\item $T(G;1,2)$ is the number of spanning subgraphs of $G$ having the same number of components as $G$.
\item $T(G;1,0)$ is the number of acyclic orientations of $G$ with one predefined source vertex per component of $G$~\cite{GreeneZaslavsky}.
\end{itemize}
Other evaluations (up to a simple pre-factor) include the reliability polynomial, chromatic polynomial and partition function of the $q$-state Potts model. For a full list of evaluations see~\cite{BrylawskiOxley,Ellis-MonaghanMerino,zbMATH07553843}.

Given a graph polynomial of this type, a natural question is to determine its complexity, that is to classify the points $(a,b)$ according to whether there is a polynomial time algorithm to evaluate the polynomial at $(a,b)$ or whether the evaluation is computationally intractable.
Because of the inherent difficulties of measuring the complexity of algorithms involving arbitrary real numbers, we restrict $a$ and $b$ to being rational. This question was completely resolved in a groundbreaking paper by Jaeger, Vertigan and Welsh~\cite{JVW}. A stronger result was obtained by Vertigan and Welsh~\cite{zbMATH00563595}, who proved the theorem below.
 For $\alpha$ in $\mathbb{Q}-\{0\}$, let $H_{\alpha}=\{(x,y)\in \mathbb{Q}^2: (x-1)(y-1) = \alpha\}$, and let $H_0^x=\{(1,y): y \in \mathbb{Q}\}$ and $H_0^y = \{(x,1) : x \in \mathbb{Q}\}$. This family of hyperbolae seems to play a special role in the theory of the Tutte polynomial, both in terms of its evaluations and its complexity.

\begin{theorem}[Vertigan, Welsh]
\label{VW}
Evaluating the Tutte polynomial of a bipartite planar graph at any fixed point $(a,b)$ in the rational plane is $\#$P-hard apart from when $(a,b)$ lies on $H_1$ or $H_2$, or when $(a,b)$ equals $(-1,-1)$ or $(1,1)$, when there exists a polynomial-time algorithm.
\end{theorem}

Roughly speaking, the proof of the hardness part of this result (at least without the planar bipartite restriction) proceeds as follows.
By exploiting a result of Brylawski~\cite{Brylawski},
one first shows that for most points $(a,b)$, the existence of a polynomial time algorithm to evaluate $T(G;a,b)$ for every graph $G$
would imply the existence of a polynomial time algorithm to evaluate $T(G;x,y)$ at every point $(x,y)$ in $H_{\alpha}$, where $\alpha=(a-1)(b-1)$.
Given a graph $G$, let $G^k$ and $G_k$ denote, respectively, the graph obtained by replacing every edge of $G$ by $k$ parallel edges and the graph obtained by replacing every non-loop of $G$ by a path comprising $k$ edges and every loop by a circuit comprising $k$ edges. The former is known as the \emph{$k$-thickening} of $G$ and the latter as the \emph{$k$-stretch} of $G$. Brylawski gave expressions for the Tutte polynomials of $G^k$ and $G_k$ in terms of the Tutte polynomial of $G$. By varying $k$, one may obtain expressions for $T(G;a_k,b_k)$ at a sequence $\{(a_k,b_k)\}$ of points on $H_{\alpha}$, and then solve for the coefficients of the one-variable polynomial obtained by restricting the domain of $T$ to $H_\alpha$. There remain several special cases because the sequence $\{(a_k,b_k)\}$ sometimes contains only a small number of distinct points. The second step proceeds by determining a $\#$P-hard point on each curve $H_{\alpha}$. Many of these come from evaluations of the chromatic polynomial.

The Tutte polynomial is essentially a generating function for the number of subsets of the edges of a graph according to their rank and size.
Following the work of Jaeger, Vertigan and Welsh, many authors have established corresponding results for a variety of graph polynomials defined in a similar way but using different notions of rank. These include the cover polynomial~\cite{zbMATH06072038}, the Bollob\'as--Riordan polynomial~\cite{zbMATH05770222}, the interlace polynomial~\cite{zbMATH05623498}, the rank generating function of a graphic $2$-polymatroid~\cite{zbMATH05039067} and the Tutte polynomial of a bicircular matroid~\cite{zbMATH05039064}. In each case, the proof techniques have some similarities: the bulk of the work is done using a graph operation analogous to the thickening, but there are considerable technical difficulties required to deal with the special cases and to complete the proof. These results provide evidence for Makowsky's Difficult Point Conjecture which states that for an $n$-variable graph polynomial $P$ that may be defined in monadic second order logic, there is a set $S$ of points with the following properties:
\begin{enumerate}
\item For every $\mathbf x\in S$, there is a polynomial time algorithm to evaluate $P(\mathbf x)$;
\item For every $\mathbf x\notin S$, it is $\#$P-hard to evaluate $P(\mathbf x)$;
\item The set $S$ is the finite union of algebraic sets in $\comps^n$ each having dimension strictly less than $n$.
\end{enumerate}
For full details see~\cite{zbMATH05552168}.

In this paper we prove results analogous to Theorem~\ref{VW} for two graph polynomials, the Tutte polynomials of a rooted graph and a rooted digraph, and a polynomial of binary matrices, the Tutte polynomial of a binary greedoid.
Each of these polynomials is a special case of the Tutte polynomial of a greedoid introduced by Gordon and McMahon~\cite{GordonMcMahon} and the proofs have considerable commonality. (All the necessary definitions are provided in the next sections.) The graph polynomials are the analogue of the Tutte polynomial for rooted graphs and rooted digraphs, and our results provide further evidence
for Makowsky's Difficult Point Conjecture.

An overview of the paper is as follows. In Section~\ref{Rooted Graphs, Rooted Digraphs and Greedoids} we provide necessary background on rooted graphs, rooted digraphs, greedoids and computational complexity. In the following section we describe the Tutte polynomial of a greedoid and list some of its evaluations for each of the three classes of greedoid that we work with.
Within our hardness proofs we require an analogue of the thickening operation and various other constructions which can be defined for arbitrary greedoids, and may be of independent interest. We describe these in Section~\ref{sec:constructions} and provide analogues of Brylawski's results~\cite{Brylawski} expressing the Tutte polynomial for these constructions in terms of the Tutte polynomials of their constituent greedoids.

In Section~\ref{section rooted graphs hardness}, we prove the following result completely determining the complexity of evaluating the Tutte polynomial of a rooted graph at a rational point.
\begin{theorem}
\label{maintheoremrootedgraph}
Evaluating the Tutte polynomial of a connected, rooted, planar, bipartite graph at any fixed point $(a,b)$ in the rational $xy$-plane is $\#$P-hard apart from when $(a,b)$ equals $(1,1)$ or when $(a,b)$ lies on $H_1$.

There are polynomial time algorithms to evaluate the Tutte polynomial of a rooted graph at $(1,1)$ and at any point lying on $H_1$.
\end{theorem}

In Section~\ref{sec:rooted digraph}, we prove the equivalent result for the Tutte polynomial of a rooted digraph.
\begin{theorem}
\label{maintheoremdigraph}
Evaluating the Tutte polynomial of a root-connected, rooted digraph at any fixed point $(a,b)$ in the rational $xy$-plane is $\#$P-hard apart from when $(a,b)$ equals $(1,1)$, when $(a,b)$ lies on $H_1$, or when $b=0$.

There are polynomial time algorithms to evaluate the Tutte polynomial of a rooted digraph at $(1,1)$, at any point lying on $H_1$ and at any point $(a,0)$.
\end{theorem}

We then determine the complexity of evaluating the Tutte polynomial of a binary greedoid.
\begin{theorem}
\label{maintheorembinarygreedoid}
Evaluating the Tutte polynomial of a binary greedoid at any fixed point $(a,b)$ in the rational $xy$-plane is $\#$P-hard apart from when $(a,b)$ lies on $H_1$.

There is a polynomial time algorithm to evaluate the Tutte polynomial of a binary greedoid at any point lying on $H_1$.
\end{theorem}

One special case of this theorem depends on a special case of an unpublished result of Vertigan, who proved that the problem of counting bases of a binary matroid is $\#$P-complete. For completeness, in Appendix~\ref{sec:appendix}, we provide a proof of this result for all fields.

\section{Preliminaries}
\label{Rooted Graphs, Rooted Digraphs and Greedoids}

\subsection{Rooted graphs and digraphs}All our graphs are allowed to have loops and multiple edges.
A \emph{rooted graph} is a graph with a distinguished vertex called the \emph{root}. Most of the graphs we work with will be rooted but occasionally we will work with a graph without a root. For complete clarity, we will sometimes refer to such graphs as \emph{unrooted graphs}.
We denote a rooted graph $G$ with vertex set $V(G)$, edge set $E(G)$ and root $r(G)$ by a triple
$(V(G),E(G),r(G))$. We omit the arguments when there is no fear of ambiguity.
Many of the standard definitions for graphs can be applied to rooted graphs in the natural way. Two rooted graphs $(V,E,r)$ and $(V',E',r')$ are \emph{isomorphic} if the unrooted graphs $(V,E)$ and $(V',E')$ are isomorphic via an isomorphism mapping $r$ to $r'$.
For a subset $A$ of $E$, the \emph{rooted spanning subgraph} $G|A$ is formed from $G$ by deleting all the edges in $E-A$ (and keeping all the vertices). The \emph{root component} of $G$ is the connected component of $G$ containing the root.
A set $A$ of edges of $G$ is \emph{feasible} if the root component of $G|A$ is a tree and contains every edge of $A$. We define the \emph{rank} $\rho_G(A)$ of $A$ to be \[\rho_G(A) =\max\{|A'|: A'\subseteq A, A \text{ is feasible}\}.\] We omit the subscript when the context is clear. We let $\rho(G) = \rho(E)$. Observe that a set $A$ of edges is feasible if and only if $\rho(A) = |A|$. A feasible set is a \emph{basis} if $\rho(A)=\rho(G)$.
So $A$ is a basis of $G$ if and only if it is the edge set of a spanning tree of the root component of $G$.

A \emph{rooted digraph} is a digraph with a distinguished vertex called the \emph{root}.
We denote a rooted digraph $D$ with vertex set $V(D)$, edge set $E(D)$ and root $r(D)$ by a triple
$(V(D),E(D),r(D))$.
Once again we omit the arguments when there is no chance of ambiguity. Two rooted digraphs $(V,E,r)$ and $(V',E',r')$ are \emph{isomorphic} if the unrooted digraphs $(V,E)$ and $(V',E')$ are isomorphic via an isomorphism mapping $r$ to $r'$.
We say that the \emph{underlying rooted graph} of a rooted digraph is the rooted graph we get when we remove all the directions on the edges.
For a subset $A$ of $E$, the \emph{rooted spanning subdigraph} $D|A$ is formed from $D$ by deleting all the edges in $E-A$.
The \emph{root component} of $D$ is formed by deleting every vertex $v$ to which there is no directed path from $r$ in $D$, together with its incident edges.
The rooted digraph is \emph{root-connected} if there is a directed path from the root to every other vertex.
The rooted digraph $D$ is an \emph{arborescence rooted at $r$} if $D$ is root-connected and its underlying rooted graph is a tree.
Observe that a set $A$ of edges of $D$ is \emph{feasible} if and only if the root
component of $D|A$ is an arborescence rooted at $r$ and contains every edge of $A$. The \emph{rank} $\rho_D(A)$ of $A$ is defined by
\[\rho_D(A) = \max\{|A'|: A' \subseteq A, D|A' \text{ is feasible}\}.\] We can omit the subscript when the context is clear. We let $\rho(D) = \rho(E)$.
A set $A$ of edges is feasible if and only if $\rho(A) = |A|$. A feasible set is a \emph{basis} if $\rho(A) = \rho(D)$.
So $A$ is a basis of $D$ if and only if it is the edge set of an arborescence rooted at $r$ that includes every vertex of the root component of $D$.

\subsection{Greedoids}

Greedoids are generalizations of matroids, first introduced by Korte and Lov\'{a}sz in 1981 in~\cite{KorteLovasz}. The aim was to generalize the characterization of matroids as hereditary set systems on which the greedy algorithm is guaranteed to determine the optimal member of the set system, according to an arbitrary weighting. Most of the information about greedoids which we summarise below can be found in~\cite{bjorner+zieglerzbMATH00067326} or~\cite{kortebookzbMATH04212062}.

\begin{definition}
A \emph{greedoid} $\Gamma$ is an ordered pair $(E,\mathcal{F})$ consisting of a finite set $E$ and a non-empty collection $\mathcal{F}$ of subsets of $E$ satisfying the following axioms:
\begin{enumerate}[label=(G\arabic*),align=left,leftmargin=*]
\item $\emptyset \in \mathcal{F}$.
\item \label{ax:G2} For all $F$ and $F'$ in $\mathcal{F}$ with $|F'|<|F|$ there exists some $x \in F-F'$ such that $F' \cup x \in \mathcal{F}$.
\end{enumerate}
\end{definition}
The set $E$ is \emph{ground set} of $\Gamma$ and the members of $\mathcal{F}$ are the \emph{feasible sets} of $\Gamma$. The axioms are the first and third of the usual axioms specifying a matroid in terms of its independent sets, so clearly every matroid is a greedoid, but a greedoid does not necessarily satisfy the hereditary property satisfied by the independent sets of a matroid requiring that the
collection of independent sets is closed under taking subsets.
The \emph{rank} $\rho_{\Gamma}(A)$  of a subset $A$ of $E$ is given by \[\rho_{\Gamma}(A) = \max\{|A'|: A' \subseteq A, A' \in \mathcal{F}\}\] and we let $\rho(\Gamma)= \rho_{\Gamma}(E)$.
We omit the subscript when the context is clear.
Notice that a set $A$ is feasible if and only if $\rho(A) = |A|$. A feasible set is a \emph{basis} if $\rho(A) = \rho(\Gamma)$. We denote the collection of bases of $\Gamma$ by $\mathcal B(\Gamma)$.
Axiom~\ref{ax:G2} implies that every basis has the same cardinality. Note that the rank function determines $\Gamma$ but the collection of bases does not. For example, suppose that a greedoid has ground set $\{1,2\}$ and unique basis $\{1,2\}$. Then its collection of feasible sets could either be $\{\emptyset, \{1\}, \{1,2\}\}$, $\{\emptyset, \{2\}, \{1,2\}\}$ or $\{\emptyset, \{1\}, \{2\}, \{1,2\}\}$.

The rank function of a greedoid can be characterized in a similar way to the rank function of a matroid~\cite{KorteLovasz3}.
\begin{proposition}
The rank function $\rho$ of a greedoid with ground set $E$ takes integer values and satisfies each of the following.
\begin{enumerate}[label=(GR\arabic*),align=left,leftmargin=*]
\item For every subset $A$ of $E$, $0\leq \rho(A) \leq |A|$;
\item \label{ax:GR2} For all subsets $A$ and $B$ of $E$ with $A\subseteq B$, $\rho(A)\leq \rho(B)$;
\item For every subset $A$ of $E$, and elements $e$ and $f$, if $\rho(A)=\rho(A\cup e)=\rho(A \cup f)$, then $\rho(A)=\rho(A\cup e\cup f)$.
\end{enumerate}
Moreover if $E$ is a finite set and $\rho$ is a function from the subsets of $E$ to the integers, then $\rho$ is the rank function of a greedoid with ground set $E$ if and only if $\rho$ satisfies conditions (GR1)--(GR3) above.
\end{proposition}

The following lemma is easily proved using induction on $|B|$ and will be useful later.
\begin{lemma}\label{lem:rankuseful}
Let $(E,\rho)$ be a greedoid specified by its rank function and let $A$ and $B$ be subsets of $E$ such that for all $b\in B$, $\rho (A\cup b)=\rho(A)$. Then $\rho(A\cup B)=\rho(A)$.
\end{lemma}

Two greedoids $\Gamma_1 = (E_1, \mathcal{F}_1)$ and $\Gamma_2 = (E_2, \mathcal{F}_2)$ are \emph{isomorphic}, denoted by $\Gamma_1 \cong \Gamma_2$, if there exists a bijection $f: E_1 \rightarrow E_2$ that preserves the feasible sets.

The following two examples of greedoids were introduced in~\cite{zbMATH03871387}.
Let $G$ be a rooted graph. Take $\Gamma=(E,\mathcal{F})$ so that $E=E(G)$ and a subset $A$ of $E$ is in $\mathcal{F}$ if and only if the root component of $G|A$ is a tree containing every edge of $A$. Then $\Gamma$ is a greedoid.
Any greedoid which is isomorphic to a greedoid arising from a rooted graph in this way is called a \emph{branching greedoid}. The branching greedoid of a rooted graph $G$ is denoted by $\Gamma(G)$.

Similarly suppose we have a rooted digraph $D$ and take $\Gamma=(E,\mathcal{F})$ so that $E={E}(D)$ and a subset $A$ of $E$ is in $\mathcal{F}$ if and only if the root component of $D|A$ is an arborescence rooted at $r$ and contains every edge of $A$. Then $\Gamma$ is a greedoid.
Any greedoid which is isomorphic to a greedoid arising from a rooted digraph in this way is called a \emph{directed branching greedoid}. The directed branching greedoid of a rooted digraph $D$ is denoted by $\Gamma(D)$. (There should be no ambiguity with the overload of notation for a branching greedoid and a directed branching greedoid.)

Notice that for both rooted graphs and digraphs, the concepts of feasible set, basis and rank are compatible with their definitions for the associated branching greedoid or directed branching greedoid in the sense that a set $A$ of edges is feasible in a rooted graph $G$ if and only if it is feasible in $\Gamma(G)$, and similarly for the other concepts.

We now define the class of \emph{binary greedoids}. These are a special case of a much broader class, the \emph{Gaussian elimination greedoids}, introduced by Goecke in~\cite{Goecke}, motivated by the Gaussian elimination algorithm.
Let $M$ be an $m\times n$ binary matrix.
It is useful to think of the rows and columns of $M$ as being labelled by the elements of $[m]$ and $[n]$ respectively, where $[n]=\{1,\ldots,n\}$. If $X$ is a subset of $[m]$ and $Y$ is a subset of $[n]$, then $M_{X,Y}$ denotes the matrix obtained from $M$ by deleting all the rows except those with labels in $X$ and all the columns except those with labels in $Y$. Take $\Gamma=([n],\mathcal F)$, so that
\[\mathcal{F} = \{A \subseteq [n]: \text{ the submatrix $M_{[|A|],A}$ is non-singular}\}.\]
By convention, the empty matrix is considered to be non-singular.
Then $\Gamma$ is a greedoid. Any greedoid which is isomorphic to a greedoid arising from a binary matrix in this way is called a \emph{binary greedoid}. The binary greedoid of a binary matrix $M$ is denoted by $\Gamma(M)$.

\begin{example}
Let \[M = \bordermatrix{~ & 1 & 2 & 3 & 4 \cr
                        ~ & 1 & 0 & 0 & 1 \cr
                        ~ & 1 & 0 & 1 & 0 \cr
                        ~ & 0 & 1 & 1 & 1 \cr}.\]
The binary greedoid $\Gamma(M)$ has ground set $\{1,2,3,4\}$ and feasible sets \[\{\emptyset, \{1\}, \{4\}, \{1,3\}, \{1,4\}, \{3,4\}, \{1,2,3\}, \{1,2,4\},\{2,3,4\}\}.\]
\end{example}

The following lemma is clear.
\begin{lemma}
Let $E=[n]$, let $M$ be an $m \times n$ binary matrix with columns labelled by $E$ and let $M'$ be obtained from $M$ by adding row $i$ to row $j$, where $i < j$. Then $\Gamma(M') \cong \Gamma(M)$.
\end{lemma}
A consequence of this lemma is that if $\Gamma$ is a binary greedoid, then there is a binary matrix $M$ with linearly independent rows so that $\Gamma=\Gamma(M)$. With this in mind we easily obtain the following result which will be useful later.

\begin{lemma}
\label{binary greedoid to matroid}
Let $\Gamma$ be a binary greedoid. Then there is a binary matroid $M$ so that $\mathcal B(M)=\mathcal B(\Gamma)$.
\end{lemma}

In contrast with the situation in matroids, where every graphic matroid is binary, it is not the case that every branching greedoid is binary. For example, take $G$ to be the star with four vertices in which the central vertex is the root. Then $\Gamma(G)$ is not binary. The same example but with the edges directed away from the root demonstrates that not every directed branching greedoid is binary.

An element of a greedoid is a \emph{loop} if it does not belong to any feasible set. So if $G$ is a rooted graph then an edge $e$ is a loop of $\Gamma(G)$ if it does not lie on any path from the root and if $G$ is connected then it is just a loop in the normal graph-theoretic sense.
Similarly if $D$ is a directed rooted graph then an edge $e$ is a loop of $\Gamma(D)$ if it does not lie on any directed path from the root.
As the concepts of loops in greedoids and in rooted graphs and digraphs do not completely coincide, we use the term \emph{greedoid loop} whenever there is potential for confusion.

Let $\Gamma$ be a greedoid with ground set $E$ and rank function $\rho$. Elements $e$ and $f$ of $E$ are said to be \emph{parallel} in $\Gamma$ if for all subsets $A$ of $E$,
\[\rho(A \cup e) = \rho(A \cup f) = \rho(A \cup e \cup f).\]
As far as we are aware, the following elementary lemma does not seem to have been stated before.
\begin{lemma}
Let $\Gamma$ be a greedoid. Define a relation $\bowtie$ on the ground set of $\Gamma$ by $e\bowtie f$ if $e$ and $f$ are parallel in $\Gamma$. Then $\bowtie$ is an equivalence relation and if $\Gamma$ has at least one loop, then one of the equivalence classes of $\bowtie$ comprises the set of loops.
\end{lemma}

\begin{proof}
The only part of the lemma that is not immediately obvious is that $\bowtie$ is transitive. Let $\rho$ be the rank function of $\Gamma$ and $e$, $f$ and $g$ be elements of $\Gamma$, so that
$e\bowtie f$ and $f\bowtie g$. Then for any subset $A$ of elements of $\Gamma$, we have $\rho(A\cup e)=\rho (A\cup f)=\rho(A\cup e \cup f)$ and $\rho(A\cup f)=\rho(A\cup g)=\rho(A\cup f \cup g)$. Thus $\rho(A\cup e)=\rho(A\cup g)$. By applying Lemma~\ref{lem:rankuseful} to $A\cup f$ and elements $e$ and $g$, we see that $\rho(A\cup e\cup f \cup g) = \rho(A\cup f)$. Thus, by~\ref{ax:GR2}, $\rho(A\cup f) = \rho(A\cup e\cup f\cup g) \geq \rho(A\cup e \cup g) \geq \rho (A\cup e)$. But as $\rho(A\cup e)=\rho(A\cup f)$, equality must hold throughout, so $\rho(A\cup e\cup g)=\rho(A\cup e)=\rho(A\cup g)$, as required.
\end{proof}

\subsection{Complexity}

We assume some familiarity with computational complexity and refer the reader to one of the standard texts such as~\cite{Garey} or~\cite{Papadimitriou} for more background. Given two computational problems $\pi_1$ and $\pi_2$, we say that $\pi_2$ is \emph{Turing reducible} to $\pi_1$ if there exists a deterministic Turing machine solving $\pi_2$ in polynomial time using an oracle for $\pi_1$, that is a subroutine returning an answer to an instance of $\pi_1$ in constant-time. When $\pi_2$ is Turing reducible to $\pi_1$ we write $\pi_2 \propto_T \pi_1$ and we say that solving problem $\pi_1$ is at least as hard as solving problem $\pi_2$. The relation $\propto_T$ is transitive.

Informally, the class $\#$P is the counting analogue of NP, that is, the class of all counting problems corresponding to decision problems in NP. Slightly more precisely, a problem is in $\#$P if it counts the number of accepting computations or ``witnesses'' of a problem in NP. Consider the decision problem of determining whether a graph has a proper vertex $3$-colouring.
The obvious non-deterministic algorithm for this problem interprets a ``witness'' as a colouring of the vertices with $3$ colours and verifies that it is a proper colouring. So the corresponding problem in $\#$P would be to determine the number of proper vertex $3$-colourings.
A computational problem $\pi$ is said to be \emph{$\#$\text{P}-hard} if $\pi' \propto_T \pi$ for all $\pi' \in \#\text{P}$, and \emph{$\#$P-complete} if, in addition, $\pi \in \#\text{P}$.
Counting the number of vertex $3$-colourings of a graph is an example of an
$\#$P-complete problem.

The following lemma is crucial in many of our proofs.
\begin{lemma}\label{lem:gauss}
There is an algorithm which when given a non-singular integer $n \times n$ matrix $A$ and an integer $n$-vector $b$ such that the absolute value of every entry of $A$ and $b$ is at most $2^l$, outputs the vector $x$ so that $Ax=b$, running in time bounded by a polynomial in $n$ and $l$.
\end{lemma}
One algorithm to do this is a variant of Gaussian elimination known as the Bareiss algorithm~\cite{zbMATH03298216}. Similar ideas were presented by Edmonds~\cite{Edmonds}. See also~\cite{zbMATH00467138}.

\section{The Tutte Polynomial of a Greedoid}
\label{The Tutte Polynomial of a Rooted Graph and of a Rooted Digraph}

Extending the definition of the Tutte polynomial of a matroid, McMahon and Gordon defined the Tutte polynomial of a greedoid in~\cite{GordonMcMahon}.
The \emph{Tutte polynomial} of a greedoid $\Gamma$ with ground set $E$ and rank function $\rho$ is given by
\[ T(\Gamma;x,y) = \sum_{A\subseteq E}(x-1)^{\rho(\Gamma)-\rho(A)}(y-1)^{|A|-\rho(A)}.\]
When $\Gamma$ is a matroid, this reduces to the usual definition of the Tutte polynomial of a matroid.
For a rooted graph $G$ we let $T(G;x,y) = T(\Gamma(G);x,y)$, for a rooted digraph $D$ we let $T(D;x,y)=T(\Gamma(D);x,y)$ and for a binary matrix $M$ we let $T(M;x,y)=T(\Gamma(M);x,y)$.

\begin{example}\label{eg:littletrees}\mbox{ }
\begin{enumerate}
\item Let
$P_k$ be the rooted (undirected) path with $k$ edges in which the root is one of the leaves. Then
\[T(P_k;x,y) = 1+ \sum_{i=1}^k (x-1)^i y^{i-1}.\]

\item Let
$S_k$ be the rooted (undirected) star with $k$ edges in which the root is the central vertex. Then
\[T(S_k;x,y) = x^k.\]
\end{enumerate}
\end{example}

The Tutte polynomial of a greedoid retains many of the properties of the Tutte polynomial of a matroid, for example, it has a delete--contract recurrence, although its form is not as simple as that of the Tutte polynomial of a matroid~\cite{GordonMcMahon}.
Moreover, for a greedoid $\Gamma$:
\begin{itemize}
\item $T(\Gamma;1,1)$ is the number of bases of $\Gamma$;
\item $T(\Gamma;2,1)$ is the number of feasible sets of $\Gamma$;
\item $T(\Gamma;1,2)$ is the number of subsets $A$ of elements of $\Gamma$ so that $\rho(A)=\rho(\Gamma)$.
\item $T(\Gamma;2,2)= 2^{|E(\Gamma)|}$.
\end{itemize}
But the Tutte polynomial of a greedoid also differs fundamentally from the Tutte polynomial of a matroid, for instance,
unlike the Tutte polynomial of a matroid, the Tutte polynomial of a greedoid can have negative coefficents. For example, $T(\Gamma(P_2);x,y) = x^2y-2xy+x+y$.

The Tutte polynomial of a rooted graph has some of the same evaluations as the Tutte polynomial of a unrooted graph.
Let $G$ be a rooted graph with edge set $E$.
\begin{itemize}
\item $T(G;1,1)$ is the number of spanning trees of the root component of $G$. (When $G$ is connected, this is just the number of spanning trees of $G$.)
\item $T(G;2,1)$ is the number of subsets $A$ of $E$, so that the root component of $G|A$ is a tree containing all the edges of $A$.
\item $T(G;1,2)$ is the number of subsets $A$ of $E$ so that the root component of $G|A$ includes every vertex of the root component of $G$. (When $G$ is connected, this is just the number of subsets $A$ so that $G|A$ is connected.)
\item If no component of $G$ other than the root component has edges, then
$T(G;1,0)$ is the number of acyclic orientations of $G$ with a unique source. Otherwise $T(G;1,0)=0$.
\end{itemize}

We record the following proposition stating that the Tutte polynomial of a connected rooted graph $G$ coincides with the Tutte polynomial of the corresponding unrooted graph $G'$ along the line $x=1$. This is easy to prove by noting that $\rho(G) = r(G')$ and a subset $A$ of the edges of $G$ satisfies $\rho(A)=\rho(G)$ if and only if $r(A)=r(G')$.

\begin{proposition}\label{prop:x=1}
Let $G=(V,E,r)$ be a connected rooted graph and let $G'=(V,E)$ be the corresponding unrooted graph. Then
\[T(G;1,y) = T(G';1,y).\]
\end{proposition}

We list some evaluations of the Tutte polynomial of a digraph. Let $D$ be a rooted digraph with edge set $E$ and root $r$.
\begin{itemize}
\item $T(D;1,1)$ is the number of spanning arborescences of the root component of $D$ rooted at $r$.
(When $D$ is root-connected, this is just its number of spanning arborescences rooted at $r$.)
\item $T(D;2,1)$ is the number of subsets $A$ of $E$, so that the root component of $D|A$ is an arborescence rooted at $r$ containing every edge of $A$.
\item $T(D;1,2)$ is the number of subsets $A$ of $E$, so that the root component of $D|A$ includes every vertex of the root component of $D$. (When $D$ is connected, this is just the number of subsets $A$ so that $D|A$ is root-connected.)
\item $T(D;1,0)=1$ if $D$ is acyclic and root-connected, and $0$ otherwise.
\end{itemize}
The last evaluation will be discussed in more detail in Section~\ref{sec:rooted digraph}.

Gordon and McMahon~\cite{GordonMcMahon} proved that if $T_1$ and $T_2$ are rooted arborescences, then
$T(T_1;x,y) = T(T_2;x,y)$ if and only if $T_1 \cong T_2$.

We list some evaluations of the Tutte polynomial of a binary matroid. Let $M$ be an $m\times n$ binary matrix with linearly independent rows.
\begin{itemize}
\item $T(M;1,1)$ is the number of subsets $A$ of the columns of $M$ so that the submatrix of $M$ corresponding to the columns in $A$ is non-singular.
\item $T(M;2,1)$ is the number of subsets $A$ of the columns of $M$ so that the submatrix $M_{[|A|],A}$ is non-singular.
\item $T(M;1,2)$ is the number of subsets $A$ of the columns of $M$ containing a subset $A'$ so that the submatrix of $M$ corresponding to the columns in $A'$ is non-singular.
\end{itemize}

If a point $(a,b)$ lies on the hyperbola $H_1$ then we have $(a-1)(b-1)=1$ by definition. Thus the Tutte polynomial of a greedoid $\Gamma$
evaluated at such a point is given by
\begin{align*}
T(\Gamma;a,b) &= \sum_{A \subseteq E(\Gamma)}(a-1)^{\rho(\Gamma)-\rho(A)}(b-1)^{|A|-\rho(A)}\\&=(a-1)^{\rho(\Gamma)}\sum_{A \subseteq E(\Gamma)}\left(\frac{1}{a-1}\right)^{|A|}=(a-1)^{\rho(\Gamma)-|E(\Gamma)|}a^{|E(\Gamma)|}.
\end{align*}
Therefore, given $|E(\Gamma)|$ and $\rho(\Gamma)$, it is easy to compute $T(\Gamma;a,b)$ in polynomial time. For all of the greedoids that we consider, both $|E(\Gamma)|$ and $\rho(\Gamma)$ will be either known or easily computed.

The \emph{characteristic polynomial} of a greedoid was first introduced by Gordon and McMahon in~\cite{GordonMcMahon2} and is a generalization of the characteristic or chromatic polynomial of a matroid. For a greedoid $\Gamma$, the \emph{characteristic polynomial} $p(\Gamma;\lambda)$ is defined by \begin{equation}\label{characteristic polynomial specialization of tutte}p(\Gamma;\lambda) = (-1)^{\rho(\Gamma)} T(\Gamma;1-\lambda,0).\end{equation}

\section{Greedoid Constructions}
\label{sec:constructions}
In this section we introduce three greedoid constructions and give expressions for the Tutte polynomial of greedoids resulting from these constructions.

The first construction is just the generalization of the $k$-thickening operation introduced by Brylawski~\cite{Brylawski} from matroids to greedoids.
Given a greedoid $\Gamma=(E,\mathcal F)$, its $k$-thickening is the greedoid $\Gamma^k$ that, informally speaking, is formed from $\Gamma$ by replacing each element by $k$ parallel elements. More precisely, $\Gamma^k$ has ground set $E'= E \times [k]$ and collection $\mathcal F'$ of feasible sets as follows. Define $\mu$ to be the projection operator $\mu: 2^{E \times [k]} \rightarrow 2^E$ so that element $e\in \mu(A)$ if and only if $(e,i) \in A$ for some $i$. Now a subset $A$ is feasible in $\Gamma^k$ if and only if $\mu(A)$ is feasible in $\Gamma$ and $|\mu(A)|=|A|$. The latter condition ensures that $A$ does not contain more than one element replacing a particular element of $\Gamma$.

It is clear that $\Gamma^k$ is a greedoid and moreover $\rho_{\Gamma^k}(A)= \rho_{\Gamma}(\mu(A))$. In particular $\rho(\Gamma^k)=\rho(\Gamma)$. For any element $e$ of $\Gamma$ the elements $(e,i)$ and $(e,j)$ are parallel. The effect of the $k$-thickening operation on the Tutte polynomial of a greedoid is given in the following theorem, generalizing the expression for the $k$-thickening of the Tutte polynomial given by Brylawski~\cite{Brylawski}.

\begin{theorem}
\label{greedoid thickening}
Let $\Gamma$ be a greedoid. The Tutte polynomial of the $k$-thickening $\Gamma^k$ of $\Gamma$ when $y \neq -1$ is given by
\begin{equation}
\label{greedoid thickening equation when y not equal -1}
T(\Gamma^k;x,y) = (1+y+ \cdots +y^{k-1})^{\rho_G(\Gamma)}T\left(\Gamma;\frac{x+y+ \cdots +y^{k-1}}{1+y+ \cdots + y^{k-1}},y^k\right).\end{equation}
When $y=-1$ we have
\[T(\Gamma^k;x,-1) = \begin{cases} (x-1)^{\rho_G(\Gamma)} & \text{if $k$ is even; }\\ T(\Gamma;x,-1) & \text{if $k$ is odd. }
\end{cases}\]
\end{theorem}

\begin{proof}
Let $\Gamma^k$ be the $k$-thickened greedoid, let $E'$ denote its ground set and let $E$ be the ground set of $\Gamma$.
Then $E'=E \times [k]$. Let $\mu$ be the mapping defined in the discussion at the beginning of this section.
To ensure that we do not divide by zero in our calculations, we prove the case when $y=1$ separately.

For each $A' \subseteq E'$ we have $\rho_{\Gamma^k}(A') = \rho_{\Gamma}(\mu(A'))$ and furthermore $\rho(\Gamma^k) = \rho(\Gamma)$. The Tutte polynomial of $\Gamma^k$ when $y \notin\{-1,1\}$ is thus given by
\begin{align}
T(\Gamma^k;x,y) &= \sum_{A' \subseteq E'}(x-1)^{\rho(\Gamma^k)-\rho_{\Gamma^k}(A')}(y-1)^{|A'|-\rho_{\Gamma^k}(A')} \notag\\
&=\sum_{A \subseteq E}\sum_{\substack{A' \subseteq E': \\ \mu(A') = A}}(x-1)^{\rho(\Gamma)-\rho_{\Gamma}(\mu(A'))}(y-1)^{|A'|-\rho_{\Gamma}(\mu(A'))} \label{equation in greedoid thickening proof}\\
&=\sum_{A \subseteq E}(x-1)^{\rho(\Gamma)-\rho_{\Gamma}(A)}(y-1)^{-\rho_{\Gamma}(A)}\sum_{\substack{A' \subseteq E': \\ \mu(A')=A}}(y-1)^{|A'|}\notag\\
&=\sum_{A \subseteq E}(x-1)^{\rho(\Gamma)-\rho_{\Gamma}(A)}(y-1)^{-\rho_{\Gamma}(A)}(y^k-1)^{|A|}\notag\\
&=(1+y+\cdots + y^{k-1})^{\rho(\Gamma)}\sum_{A \subseteq E} \left(\frac{(x-1)(y-1)}{y^k-1}\right)^{\rho(\Gamma)-\rho_{\Gamma}(A)}(y^k-1)^{|A|-\rho_{\Gamma}(A)}\notag\\
&=(1+y+ \cdots + y^{k-1})^{\rho(\Gamma)}T\left(\Gamma; \frac{x+y+\cdots + y^{k-1}}{1+y+ \cdots + y^{k-1}},y^k\right).\notag
\end{align}

When $y=1$ we get non-zero terms in Equation~\ref{equation in greedoid thickening proof} if and only if $|A'|=\rho_{\Gamma}(\mu(A'))$, which implies that $|A'|=|A|$. For each $A \subseteq E$ there are $k^{|A|}$ choices for $A'$  such that $\mu(A')=A$ and $|A'|=|A|$. Therefore we have
\begin{align*}
T(\Gamma^k;x,1) &= \sum_{\substack{A \subseteq E: \\ \rho_{\Gamma}(A) = |A|}}(x-1)^{\rho(\Gamma)-\rho_{\Gamma}(A)} \sum_{\substack{A' \subseteq E': \\ \mu(A') = A, \\ |A'|=|A|}}1 = \sum_{\substack{A \subseteq E: \\ \rho_{\Gamma}(A) = |A|}}(x-1)^{\rho(\Gamma)-\rho_{\Gamma}(A)} k^{\rho_{\Gamma}(A)}\\
&= \sum_{\substack{A \subseteq E: \\ \rho_{\Gamma}(A) = |A|}}\left(\frac{x-1}{k}\right)^{\rho(\Gamma)-\rho_{\Gamma}(A)} k^{\rho(\Gamma)}= k^{\rho(\Gamma)} T\left(\Gamma;\frac{x+k-1}{k},1\right)
\end{align*}
which agrees with Equation~\ref{greedoid thickening equation when y not equal -1} when $y=1$.

When $y=-1$ we have
\begin{align}
T(\Gamma^k;x,-1) &=\sum_{A \subseteq E}\sum_{\substack{A' \subseteq E': \\ \mu(A') = A}}(x-1)^{\rho(\Gamma)-\rho_{\Gamma}(\mu(A'))}(-2)^{|A'|-\rho_{\Gamma}(\mu(A'))} \notag\\
&=\sum_{A \subseteq E}(x-1)^{\rho(\Gamma)-\rho_{\Gamma}(A)}(-2)^{-\rho_{\Gamma}(A)}\sum_{\substack{A' \subseteq E': \\ \mu(A')=A}}(-2)^{|A'|}\notag\\
&=\sum_{A \subseteq E}(x-1)^{\rho(\Gamma)-\rho_{\Gamma}(A)}(-2)^{-\rho_{\Gamma}(A)}((-1)^k-1)^{|A|}\notag\\
&= \left\{ \begin{array}{ll} (x-1)^{\rho(\Gamma)} & \text{if $k$ is even};\\ T(\Gamma;x,-1) & \text{if $k$ is odd.} \end{array} \right. \notag
\end{align}
Note that the only contribution to $T(\Gamma^k;x,-1)$ when $k$ is even is from the empty set.
\end{proof}

The second construction is a little more involved. To motivate it we first describe a natural construction operation on rooted graphs. Let $G$ and $H$ be disjoint rooted graphs with $G$ being connected. Then the \emph{$H$-attachment} of $G$, denoted by $G \sim H$, is formed by taking $G$ and $\rho(G)$ disjoint copies of $H$, and identifying each vertex of $G$ other than the root with the root vertex of one of the copies of $H$. The root of $G\sim H$ is the root of $G$. See Figure~\ref{fig:attach} for an illustration of the attachment operation.

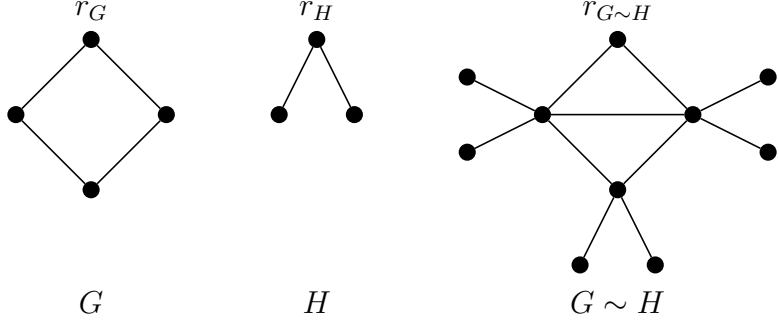
\begin{figure}[h]
\centering
\begin{tikzpicture}[label distance=0pt,semithick,v/.style={circle,fill=black,draw=black,thick,
minimum size=2pt,inner sep=2pt
},e/.style={inner sep=2pt,rectangle,text height=1.5ex,text depth=.25ex,auto}]
\node(1)[v] at (0,0) {};
\node(2)[v] at (1,1) {};
\node(3)[v,label=above:$r_{G\sim H}$] at (0,2) {};
\node(4)[v] at (-1,1) {};
\node(5)[v] at (2,1.5){};
\node(6)[v] at (2,0.5){};
\node(7)[v] at (-2,1.5){};
\node(8)[v] at (-2,0.5){};
\node(9)[v] at (-0.5,-1){};
\node(10)[v] at (0.5,-1){};
\node at (0,-1.5) {$G\sim H$};
\draw (1) to (2);
\draw (3) to (2);
\draw (3) to (4);
\draw (1) to (4);
\draw (4) to (2);
\draw (5) to (2);
\draw (6) to (2);
\draw (7) to (4);
\draw (8) to (4);
\draw (9) to (1);
\draw (10) to (1);
\node(11)[v] at (-7,0) {};
\node(12)[v] at (-6,1) {};
\node(13)[v,label=above:$r_{G}$] at (-7,2) {};
\node(14)[v] at (-8,1) {};
\node at (-7,-1.5) {$G$};
\draw (11) to (12);
\draw (13) to (12);
\draw (13) to (14);
\draw (11) to (14);
\node(23)[v] at (-4.5,1) {};
\node(22)[v] at (-3.5,1) {};
\node(21)[v,label=above:$r_{H}$] at (-4,2) {};
\node at (-4,-1.5) {$H$};
\draw (21) to (22);
\draw (23) to (21);
\end{tikzpicture}
\caption{An example of the attachment operation.\label{fig:attach}}
\end{figure}

Suppose that $V(G)=\{r,v_1,\ldots,v_{\rho(G)}\}$, where $r$ is the root of $G$, let $E_0$ be the edge set of $G$ and let $E_i$ be the edge set of the copy of $H$ attached at $v_i$. A set $F$ is feasible in $\Gamma(G\sim H)$ if and only if each of the following conditions holds.
\begin{enumerate}
\item $F \cap E_0$ is feasible in $\Gamma(G)$.
\item For all $i$ with $1 \leq i \leq \rho(G)$, $F \cap E_i$ is feasible in $\Gamma(H)$.
\item For all $i$ with $1 \leq i \leq \rho(G)$, if $v_i$ is not in the root component of
$G|(F \cap E_0)$, then $F \cap E_i = \emptyset$.
\end{enumerate}

In order to extend these ideas to general greedoids, we begin by describing the notion of a closed set, which was first defined for greedoids by Korte and Lovasz~\cite{KorteLovasz}. Let $\Gamma$ be a greedoid with ground set $E$ and rank function $\rho$. Given a subset $A$ of $E$, its \emph{closure} $\sigma_{\Gamma}(A)$ is defined by $\sigma_{\Gamma}(A)=\{e: \rho(A\cup e)=\rho(A)\}$. We will drop the dependence on $\Gamma$ whenever the context is clear.
Note that it follows from the definition that $A \subseteq \sigma(A)$.
Moreover Lemma~\ref{lem:rankuseful} implies that $\rho(\sigma(A))=\rho(A)$. Furthermore if $e \notin \sigma(A)$, then $\rho(A \cup e) > \rho(A)$, so axiom~\ref{ax:GR2} implies that $\rho(\sigma(A) \cup e)> \rho(\sigma(A))$ and hence $\sigma(\sigma(A))=\sigma(A)$. A subset $A$ of $E$ satisfying $A=\sigma(A)$ is said to be \emph{closed}. Every subset of $E$ of the form $\sigma(X)$ for some $X$ is closed.

We now introduce what we call an attachment function. Let $\Gamma$ be a greedoid with rank function $\rho$. A function $f:\mathcal{F} \rightarrow 2^{[\rho(\Gamma)]}$ is called a \emph{$\Gamma$-attachment function} if it satisfies both of the following.
\begin{enumerate}
\item For each feasible set $F$, we have $|f(F)|=\rho(F)$.
\item If $F_1$ and $F_2$ are feasible sets and $F_1 \subseteq \sigma(F_2)$ then $f(F_1) \subseteq f(F_2)$.
\end{enumerate}

The following property of attachment functions is needed later.
\begin{lemma}\label{lem:attachment}
Let $\Gamma$ be a greedoid and $f$ be a $\Gamma$-attachment function. Let $A$ be a subset of the elements of $\Gamma$ and let $F_1$ and $F_2$ be maximal feasible subsets of $A$. Then $f(F_1)=f(F_2)$.
\end{lemma}

\begin{proof}
It follows from the axioms for the feasible sets of a greedoid that all maximal feasible subsets of $A$ have the same size. Thus $\rho(F_1)=\rho(F_2)=\rho(A)$.
For every element $e$ of $A$, $\rho (F_1) \leq \rho(F_1\cup e) \leq \rho(A)$. As $\rho(F_1)=\rho(A)$, equality must hold throughout.
Thus $e\in \sigma(F_1)$.
Hence $A\subseteq \sigma(F_1)$, so $F_2\subseteq \sigma(F_1)$. By symmetry, $F_1\subseteq \sigma(F_2)$.
The result then follows from the second condition satisfied by a $\Gamma$-attachment function.
\end{proof}

Given greedoids $\Gamma_1$ and $\Gamma_2$ with disjoint ground sets, and $\Gamma_1$-attachment function $f$, we define the \emph{$\Gamma_2$-attachment of $\Gamma_1$}, denoted by $\Gamma_1 \sim_f \Gamma_2$ as follows. The ground set $E$ is the union of the ground set $E_0$ of $\Gamma_1$ together with $\rho=\rho(\Gamma_1)$ disjoint copies $E_1,\ldots, E_\rho$ of the ground set of $\Gamma_2$.
In the following we abuse notation slightly by saying that for $i>0$, a subset of $E_i$ is feasible in $\Gamma_2$ if the corresponding subset of the elements of $\Gamma_2$ is feasible.
A subset $F$ of $E$ is feasible if and only if each of the following conditions holds.
\begin{enumerate}
\item $F \cap E_0$ is feasible in $\Gamma_1$.
\item For all $i$ with $1 \leq i \leq \rho$, $F \cap E_i$ is feasible in $\Gamma_2$.
\item For all $i$ with $1 \leq i \leq \rho$, if $i\notin f(F \cap E_0)$ then $F \cap E_i=\emptyset$.
\end{enumerate}

\begin{proposition}
For any greedoids $\Gamma_1$ and $\Gamma_2$, and $\Gamma_1$-attachment function $f$, the
\emph{$\Gamma_2$-attachment of $\Gamma_1$} is a greedoid.
\end{proposition}

\begin{proof}
We use the notation defined above to describe the ground set of $\Gamma_1 \sim_f \Gamma_2$.
Clearly the empty set is feasible in $\Gamma_1 \sim_f \Gamma_2$. Suppose that $F_1$ and $F_2$ are feasible sets in
$\Gamma_1 \sim_f \Gamma_2$ with $|F_2|>|F_1|$. If there is an element $e$ of $F_2 \cap E_0$ which is not in $\sigma_{\Gamma_1} (F_1 \cap E_0)$ then $(F_1 \cap E_0) \cup e$ is feasible in $\Gamma_1$. Moreover $F_1 \cap E_0 \subseteq \sigma_{\Gamma_1} ((F_1 \cap E_0) \cup e)$, so $f(F_1 \cap E_0) \subseteq f((F_1 \cap E_0)\cup e)$. Consequently $F_1 \cup e$ is feasible in $\Gamma_1 \sim_f \Gamma_2$.

On the other hand, suppose that $F_2 \cap E_0 \subseteq \sigma_{\Gamma_1}(F_1 \cap E_0)$. Then $f(F_2 \cap E_0) \subseteq f(F_1 \cap E_0)$.
Moreover, as there is no element $e$ of $(F_2 \cap E_0)-(F_1 \cap E_0)$ such that $(F_1 \cap E_0)\cup e$ is feasible, we have
$|F_2 \cap E_0|\leq |F_1 \cap E_0|$. So for some $i$ in $f(F_2 \cap E_0)$, we have $|F_2 \cap E_i| > |F_1 \cap E_i|$. Thus there exists $e \in (F_2-F_1)\cap E_i$ such that $(F_1 \cap E_i) \cup e$ is feasible in $\Gamma_2$. As $i\in f(F_2 \cap E_0)$, we have $i \in f(F_1 \cap E_0)$.
Hence $F_1 \cup e$ is feasible in $\Gamma_1 \sim_f \Gamma_2$.
\end{proof}

Every greedoid $\Gamma$ has an attachment function formed by setting $f(F)=[|F|]$ for each feasible set $F$. However there are other examples of attachment functions. Let $G$ be a connected rooted graph in which the vertices other than the root are labelled $v_1,\ldots, v_{\rho}$. There is an attachment function $f$ defined on $\Gamma(G)$ as follows. For every feasible set $F$, define $f(F)$ so that $i\in f(F)$ if and only if $v_i$ is in the root component of $G|F$. It is straightforward to verify that $f$ is indeed an attachment function. Furthermore if $H$ is another rooted graph then $\Gamma(G\sim H)= \Gamma(G)\sim_f \Gamma(H)$.

We now consider the rank function of $\Gamma = \Gamma_1 \sim_f \Gamma_2$. We keep the same notation as above for the elements of $\Gamma$. Let $A$ be a subset of $E(\Gamma)$ and let $F$ be a maximal feasible subset of $A\cap E_0$. Then
\begin{equation} \label{eq:rankattach}  \rho_{\Gamma}(A) = \rho_{\Gamma_1}(A \cap E_0) + \sum_{i \in f(F)} \rho_{\Gamma_2}(A \cap E_i).\end{equation}
Observe that the number of subsets of $E(\Gamma)$ with specified rank, size and intersection with $E_0$ does not depend on the choice of $f$. Consequently the Tutte polynomial of $\Gamma_1 \sim_f \Gamma_2$ does not depend on $f$. We now make this idea more precise by establishing an expression for the Tutte polynomial of an attachment.

\begin{theorem}
\label{attachment function greedoids}
Let $\Gamma_1$ and $\Gamma_2$ be greedoids, and let $f$ be an attachment function for $\Gamma_1$. Then the Tutte polynomial of $\Gamma_1 \sim_f \Gamma_2$ is given by
\[ T(\Gamma_1 \sim_f \Gamma_2;x,y) = T(\Gamma_2; x,y)^{\rho(\Gamma_1)} T\Big(\Gamma_1;\frac{(x-1)^{\rho(\Gamma_2)+1}y^{|E(\Gamma_2)|}}{T(\Gamma_2;x,y)}+1,y\Big),\]
providing $T(\Gamma_2;x,y)\ne 0$.
\end{theorem}
\begin{proof}
Let $\Gamma = \Gamma_1 \sim_f \Gamma_2$.
We use the notation defined above to describe the ground set of $\Gamma$. It is useful to extend the definition of the attachment function $f$ to all subsets of $E_0$ by setting $f(A)$ to be equal to $f(F)$ where $F$ is a maximal feasible set of $A$.
Lemma~\ref{lem:attachment} ensures that extending $f$ in this way is well-defined.
It follows from Equation~\ref{eq:rankattach} that $\rho(\Gamma) = \rho(\Gamma_1)(\rho(\Gamma_2)+1)$. We have
\begin{align*}
T(\Gamma;x,y) &= \sum_{A\subseteq E(\Gamma)} (x-1)^{\rho(\Gamma)-\rho_{\Gamma}(A)} (y-1)^{|A|-\rho(A)}\\
&= \sum_{A_0 \subseteq E_0} (x-1)^{\rho(\Gamma_1) - \rho_{\Gamma_1}(A_0)}(y-1)^{|A_0|-\rho_{\Gamma_1}(A_0)}\cdot \prod_{i\notin f(A_0)} \sum_{A_i \subseteq E_i} (x-1)^{\rho(\Gamma_2)} (y-1)^{|A_i|}\\& \phantom{=}\ \cdot \prod_{i \in f(A_0)} \sum_{A_i \subseteq E_i} (x-1)^{\rho(\Gamma_2)- \rho_{\Gamma_2}(A_i)} (y-1)^{|A_i|-\rho_{\Gamma_2}(A_i)} \\
&= \sum_{A_0 \subseteq E_0} (x-1)^{\rho(\Gamma_1) - \rho_{\Gamma_1}(A_0)} (T(\Gamma_2;x,y))^{\rho_{\Gamma_1}(A_0)}\\
& \phantom{=} \ \cdot
\big((x-1)^{\rho(\Gamma_2)}
y^{|E(\Gamma_2)|}\big)^{\rho(\Gamma_1)-\rho_{\Gamma_1}(A_0)} (y-1)^{|A_0|-\rho_{\Gamma_1}(A_0)}\\
&= (T(\Gamma_2;x,y))^{\rho(\Gamma_1)} \sum_{A_0 \subseteq E_0} (y-1)^{|A_0|-\rho_{\Gamma_1}(A_0)} \Big(\frac{(x-1)^{\rho(\Gamma_2)+1}y^{|E(\Gamma_2)|}}{T(\Gamma_2;x,y)}\Big)^{\rho(\Gamma_1)-\rho_{\Gamma_1}(A_0)}\\
&= T(\Gamma_2; x,y)^{\rho(\Gamma_1)} T\Big(\Gamma_1;\frac{(x-1)^{\rho(\Gamma_2)+1}y^{|E(\Gamma_2)|}}{T(\Gamma_2;x,y)}+1,y\Big).
\end{align*}
\end{proof}

The third construction is called the full rank attachment. Given greedoids $\Gamma_1=(E_1,\mathcal F_1)$ and $\Gamma_2=(E_2, \mathcal F_2)$ with disjoint ground sets, the \emph{full rank attachment of $\Gamma_2$ to $\Gamma_1$} denoted by $\Gamma_1 \approx \Gamma_2$ has ground set $E_1 \cup E_2$ and a set $F$ of elements is feasible if either of the two following conditions holds.
\begin{enumerate}
\item $F \in \mathcal F_1$;
\item $F \cap E_1 \in \mathcal F_1$, $F\cap E_2 \in \mathcal F_2$ and $\rho_{\Gamma_1}(F\cap E_1) = \rho(\Gamma_1)$.
\end{enumerate}
It is straightforward to prove that $\Gamma_1 \approx \Gamma_2$ is a greedoid.

Suppose that $\Gamma=\Gamma_1\approx \Gamma_2$ and that $A$ is a subset of $E(\Gamma)$. Then
\[ \rho(A) = \begin{cases} \rho(A \cap E_1) & \text{if $\rho(A \cap E_1)< \rho(\Gamma_1)$,}\\
\rho(A \cap E_1) + \rho(A\cap E_2) & \text{if $\rho(A \cap E_1)= \rho(\Gamma_1)$.}
\end{cases}\]
This observation enables us to prove the following identity for the Tutte polynomial.

\begin{theorem}\label{thm:fullrankattach}
Let $\Gamma_1$ and $\Gamma_2$ be greedoids, and let $\Gamma=\Gamma_1 \approx \Gamma_2$.
Let $E$, $E_1$ and $E_2$ denote the ground sets of $\Gamma$, $\Gamma_1$ and $\Gamma_2$ respectively.
Then
\[T(\Gamma_1\approx\Gamma_2;x,y)\\ = T(\Gamma_1;x,y)(x-1)^{\rho(\Gamma_2)}y^{|E_2|} + T(\Gamma_1;1,y) ( T(\Gamma_2;x,y) - (x-1)^{\rho(\Gamma_2)}y^{|E_2|}).\]
\end{theorem}
\begin{proof}
We have
\begin{align*}
\MoveEqLeft{T(\Gamma_1 \approx \Gamma_2;x,y)}\\
&= \sum_{A\subseteq E} (x-1)^{\rho(\Gamma)-\rho_{\Gamma}(A)}(y-1)^{|A|-\rho_{\Gamma}(A)}\\
&= \sum_{\substack{A_1 \subseteq E_1:\\ \rho_{\Gamma_1}(A_1)<\rho(\Gamma_1)}} (x-1)^{\rho(\Gamma_1)-\rho_{\Gamma_1}(A_1)} (y-1)^{|A_1|-\rho_{\Gamma_1}(A_1)} \sum_{A_2 \subseteq E_2} (x-1)^{\rho(\Gamma_2)} (y-1)^{|A_2|}\\
 &\phantom{=} \  { }+
\sum_{\substack{A_1 \subseteq E_1:\\ \rho_{\Gamma_1}(A_1)=\rho(\Gamma_1)}}  (y-1)^{|A_1|-\rho_{\Gamma_1}(A_1)}\sum_{A_2 \subseteq E_2} (x-1)^{\rho(\Gamma_2)-\rho_{\Gamma_2}(A_2)} (y-1)^{|A_2|-\rho_{\Gamma_2}(A_2)}\\
&= \sum_{A_1\subseteq E_1} (x-1)^{\rho(\Gamma_1)-\rho_{\Gamma_1}(A_1)} (y-1)^{|A_1|-\rho_{\Gamma_1}(A_1)} (x-1)^{\rho(\Gamma_2)}y^{|E_2|}\\
 &\phantom{=} \ { }+
\sum_{\substack{A_1 \subseteq E_1:\\
\rho_{\Gamma_1}(A_1)=\rho(\Gamma_1)}}
(y-1)^{|A_1|-\rho_{\Gamma_1}(A_1)} \\ &\phantom{=} \ \cdot
\Big(\sum_{A_2 \subseteq E_2} (x-1)^{\rho(\Gamma_2)-\rho_{\Gamma_2}(A_2)} (y-1)^{|A_2|-\rho_{\Gamma_2}(A_2)}-
(x-1)^{\rho(\Gamma_2)}y^{|E_2|}\Big)\\&= T(\Gamma_1;x,y)(x-1)^{\rho(\Gamma_2)}y^{|E_2|} +
T(\Gamma_1;1,y)\big(T(\Gamma_2;x,y) -(x-1)^{\rho(\Gamma_2)}y^{|E_2|}\big).
\end{align*}
\end{proof}
This construction will be useful later in Section~\ref{sec:binarygreedoids} when $\Gamma_1$ and $\Gamma_2$ are binary greedoids with $\Gamma_1=\Gamma(M_1)$ and $\Gamma_2=\Gamma(M_2)$, where $M_1$ has full row rank. Then $\Gamma_1 \approx \Gamma_2 = \Gamma(M)$ where $M$ has the form \[M= \left(\begin{array}{c|c} M_1 & 0 \\ \hline 0 & M_2 \end{array}\right).\]

\section{Rooted Graphs}
\label{section rooted graphs hardness}

Throughout the remainder of the paper we focus on three computational problems. Let $\mathbb G$ denote either the class of branching greedoids, directed branching greedoids or binary greedoids. Our first problem is computing all the coefficients of the Tutte polynomial for a greedoid in the class $\mathbb G$.

\prob{$\pi_1[\mathbb G]$ : $\#$\textsc{Rooted Tutte Polynomial}}{$\Gamma \in \mathbb G$.}{The coefficients of $T(\Gamma;x,y)$.}

The second problem involves computing the Tutte polynomial along a plane algebraic curve $L$. We restrict our attention to the case where $L$ is a rational curve given by the parametric equations \[ x(t) = \frac{p(t)}{q(t)} \quad \text{ and } \quad y(t) = \frac{r(t)}{s(t)},\] where $p$, $q$, $r$ and $s$ are polynomials over $\mathbb{Q}$. More precisely, we compute the coefficients of the one-variable polynomial obtained by restricting $T$ to the curve $L$.

\prob{$\pi_2[\mathbb G,L]$ : $\#$\textsc{Rooted Tutte Polynomial Along $L$}}{$\Gamma \in \mathbb G$.}{The coefficients of the rational function of $t$ given by evaluating $T(\Gamma;x(t),y(t))$ along $L$.}
Most of the time, $L$ will be one of the hyperbolae $H_{\alpha}$. We will frequently make a slight abuse of notation by writing $L=H_{\alpha}$.

The final problem
is the evaluation of the Tutte polynomial at a fixed rational point $(a,b)$.
\prob{$\pi_3[\mathbb G,a,b]$ : $\#$\textsc{Rooted Tutte Polynomial At $(a,b)$}}{$\Gamma \in \mathbb G$.}{$T(\Gamma;a,b)$.}

It is straightforward to see that for each possibility for $\mathbb G$, we have \[\pi_3[\mathbb{G},a,b] \propto_T \pi_2[\mathbb{G},H_{(a-1)(b-1)}] \propto_T \pi_1[\mathbb{G}].\]
Our results in the remainder of the paper will determine when the opposite reductions hold.

In this section we prove Theorem~\ref{maintheoremrootedgraph}. We let $\mathcal{G}$ be the class of branching greedoids of connected, rooted, planar, bipartite graphs and take $\mathbb{G}=\mathcal{G}$. It is, however, more convenient to take the input to each problem to be a connected, rooted, planar, bipartite graph rather than its branching greedoid.

We begin by reviewing the exceptional points of Theorem~\ref{maintheoremrootedgraph}.
If a point $(a,b)$ lies on the hyperbola $H_1$ then, following the remarks at the end of Section~\ref{The Tutte Polynomial of a Rooted Graph and of a Rooted Digraph}, $T(G;a,b)$ is easily computed.
We noted in Section~\ref{The Tutte Polynomial of a Rooted Graph and of a Rooted Digraph} that for a connected rooted graph $G$, $T(G;1,1)$ is equal to the number of spanning trees of $G$. That this can be evaluated in polynomial time follows from Kirchhoff's Matrix--Tree theorem~\cite{Kirchhoff}.
Hence
there are polynomial time algorithms to evaluate the Tutte polynomial of a connected rooted graph at $(1,1)$ and at any point lying on $H_1$.
It is easy to extend this to all rooted graphs because every edge belonging to a component that does not include the root is a loop in the corresponding branching greedoid.

We will now review the hard points of Theorem~\ref{maintheoremrootedgraph}.
A key step in establishing the hardness part of Theorem~\ref{maintheoremrootedgraph} for points lying on the line $y=1$ is to strengthen a result of Jerrum~\cite{Jerrum}.
Given an unrooted graph $G=(V,E)$, a \emph{subtree} of $G$ is a subgraph of $G$ which is a tree. (We emphasize that the subgraph does not have to be an induced subgraph.)
Jerrum~\cite{Jerrum} showed that the following problem is $\#$P-complete.

\prob{$\#$\textsc{Subtrees}}{Planar unrooted graph $G$.}{The number of subtrees of $G$.}
Consider the restriction of this problem to bipartite planar graphs.

\prob{$\#$\textsc{Bisubtrees}}{Bipartite, planar unrooted graph $G$.}{The number of subtrees of $G$.}

We shall show that $\#$\textsc{Bisubtrees} is $\#$P-complete.
We say that an edge of a graph $G$ is \emph{external} in a subtree $T$ of $G$ if it is not contained in $E(T)$. Let $t_{i,j}(G)$ be the number of subtrees of $G$ with $i$ external edges having precisely one endvertex in $T$ and $j$ external edges having both endvertices in $T$.

Recall that the $k$-stretch of an unrooted graph $G$ is obtained by replacing each loop by a circuit with $k$ edges and every other edge by a path of length $k$. Let $t(G)$ denote the number of subtrees of $G$.

\begin{proposition}
\label{bipartite stretch}
For every unrooted graph $G$,
the number of subtrees of the $k$-stretch $G_k$ of $G$ is given by
\[ t(G_k) = \left(\sum_{i,j \geq 0}t_{i,j}(G)k^i \binom{k+1}{2}^j\right) + \frac{k(k-1)|E|}{2}.\]
\end{proposition}

\begin{proof}
Let $E(G) = \{e_1, e_2, \ldots, e_m\}$ and let $E_t$ be the set of edges replacing $e_t$ in $G_k$ for $1 \leq t \leq m$. Thus $E(G_k) = \bigcup_{t=1}^m E_t$. We can think of the vertices of $G_k$ as being of two types: those corresponding to the vertices of $G$ and the extra ones added when $G_k$ is formed.
We construct a function $f$ that maps every subtree $T$ of $G_k$ to a graph $T'$ which is
either a subtree of $G$ or an empty graph with no vertices or edges.
We let $V(T')$ comprise all the vertices of $V(T)$ corresponding to vertices in $G$. The edge set $E(T')$ is defined so that
$e_t \in E(T')$ if and only if $E_t \subseteq E(T)$.

Let $T'$ be a subtree of $G$ with at least one vertex, $i$ external edges having precisely one endvertex in $T'$ and $j$ external edges having both endvertices in $T'$.

If $T \in f^{-1}(T')$ then it must contain all of the edges in $G_k$ that replace the edges in $E(T')$.
Suppose there is an edge $e = v_1v_2$ in $G$ that is external in $T'$ with $v_1\in V(T')$ and $v_2\notin V(T')$.
Then there are $k$ possibilities for the subset of $E_t$ appearing in $T$.
Now suppose there exists an edge $e_t= v_1v_2$ in $G$ that is external in $T'$ with $v_1, v_2 \in V'$. Then there are $\binom{k+1}{2}$ choices for the subset of $E_t$ appearing in $T$.
Therefore,
\[|f^{-1}(T_{i,j}')| = k^i \binom{k+1}{2}^j.\]

It remains to count the subtrees of $G_k$ mapped by $f$ to a graph with no vertices. Such a subtree does not contain any vertices corresponding to vertices in $G$. There are $(k-1)|E(G)|$ subtrees of $G_k$ comprising a single vertex not in $V(G)$ and no edges, and $\binom{k-1}{2}|E(G)|$ subtrees of $G_k$ with at least one edge but not containing any vertex in $V(G)$. Hence
\[t(G_k) = \left(\sum_{i,j \geq 0}t_{i,j}(G)k^i \binom{k+1}{2}^j\right) + \frac{k(k-1)}{2}|E(G)|.\]
\end{proof}

We can now show that \textsc{Bisubtrees} is $\#$P-complete.

\begin{proposition}\label{prop:jerrum}
The problem \textsc{Bisubtrees} is $\#$P-complete.
\end{proposition}

\begin{proof}
It is clear that \textsc{Bisubtrees} belongs to $\#$P.
To establish hardness, first note that $G_2, \ldots, G_{4|E(G)|+2}$ are all bipartite and may be constructed from $G$ in polynomial time.
We have $\max_{i,j \geq 0}\{i+2j: t_{i,j}(G)>0\} \leq \max_{i,j \geq 0}\{i+2j: i+j \leq |E(G)|\} =2{|E(G)|}$. Therefore, by Proposition~\ref{bipartite stretch}, $t(G_k)$ is a polynomial in $k$ of degree at most $2|E(G)|$. So we can write \[t(G_k) = \sum_{p=0}^{2|E(G)|}a_pk^p.\]
Thus, if we compute $t(G_k)$ for $k=2,\ldots, 4|E(G)|+2$, then we can apply Lemma~\ref{lem:gauss} to recover $a_i$ for all $i$
and then determine $t(G)=t(G_1)$ in polynomial time. Therefore we have shown that \textsc{Subtrees} $\propto_T$ \textsc{Bisubtrees}.
\end{proof}

We now present three propositions which together show that at most fixed rational points $(a,b)$, evaluating the Tutte polynomial of a connected, bipartite, planar, rooted graph at $(a,b)$ is just as hard as evaluating it along the curve $H_{(a-1)(b-1)}$. The $k$-thickening operation is crucial. Notice that $\Gamma(G^k)\cong (\Gamma(G))^k$, so we may apply Theorem~\ref{greedoid thickening} to obtain an expression for $T(G^k)$.
The first proposition deals with the case when $a \neq 1$ and $b \notin \{-1,0,1\}$.

\begin{proposition}
\label{main proposition 1}
Let $L = H_{\alpha}$ for some $\alpha \in \mathbb{Q} - \{0\}$. Let $(a,b)$ be a point on $L$ such that $b \notin \{-1,0\}$. Then
\[ \pi_2[\mathcal{G},L] \propto_T \pi_3[\mathcal{G},a,b].\]
\end{proposition}

\begin{proof}
For a point $(x,y)$ on $L$ we have $y\neq 1$. Therefore $z=y-1 \neq 0$ and so $\alpha / z = x-1$. Let $G$ be in $\mathcal{G}$.
Along $L$ the Tutte polynomial of $G$ has the form
\[T(G;x,y) = T(G; 1+ \alpha / z, 1+z) = \sum_{A \subseteq E(G)}\left(\frac{\alpha}{z}\right)^{\rho(G)-\rho(A)}z^{|A|-\rho(A)} = \sum_{i = -\rho(G)}^{|E(G)|}t_i z^i,\]
for some $t_{-\rho(G)}, \ldots , t_{|E(G)|}$.

We now show that we can determine all of the coefficients $t_i$ from the evaluations $T(G^k;a,b)$ for $k= 1,\ldots, |E(G)|+\rho(G)+1$
in time polynomial in $|E(G)|$. For each such $k$, $G^k$ may be constructed from $G$ in time polynomial in $|E(G)|$ and is bipartite, planar and connected.
By Theorem~\ref{greedoid thickening}, we have
\[T(G^k;a,b) = (1+b+\ldots + b^{k-1})^{\rho(G)}T\left(G; \frac{a+b+\ldots + b^{k-1}}{1+b+\ldots + b^{k-1}},b^k\right).\]

Since $b \neq -1$, we have $1+b+\ldots + b^{k-1} \neq 0$. Therefore we may compute \[T\left(G; \frac{a+b+\ldots + b^{k-1}}{1+b+\ldots + b^{k-1}},b^k\right)\] from $T(G^k;a,b)$.
The point $\left(\frac{a+b+\ldots + b^{k-1}}{1+b+\ldots + b^{k-1}},b^k\right)$ will also be on the curve $L$ since
\[\left(\frac{a+b+\ldots + b^{k-1}}{1+b+\ldots + b^{k-1}}-1\right)(b^k-1) = (a-1)(b-1).\]
As $b \notin \{-1,0,1\}$, for $k=1,2, \ldots, |E(G)|+\rho(G)+1$, the points $\left(\frac{a+b+\ldots + b^{k-1}}{1+b+\ldots + b^{k-1}},b^k\right)$
are pairwise distinct.
Therefore by evaluating $T(G^k;a,b)$ for $k=1,\ldots, |E(G)|+\rho(G)+1$, we obtain $\sum_{i = -\rho(G)}^{|E(G)|}t_iz^i$ for $|E(G)|+\rho(G)+1$ distinct values of $z$. This gives us $|E(G)|+\rho(G)+1$ linear equations for the coefficients $t_i$. The matrix of the equations is a Vandermonde matrix and clearly non-singular. So, we may apply Lemma~\ref{lem:gauss} to compute $t_i$ for all $i$ in time polynomial in $|E(G)|$.
\end{proof}

The next proposition deals with the case when $a =1$. Recall $H_0^x = \{(1,y) : y \in \mathbb{Q}\}$ and $H_0^y = \{(x,1) : x \in \mathbb{Q}\}$.

\begin{proposition}
Let $L = H_0^x$ and let $b\in \mathbb{Q}-\{-1,0,1\}$. Then
\[ \pi_2[\mathcal{G},L] \propto_T \pi_3[\mathcal{G},1,b].\]
\end{proposition}

\begin{proof}
Let $G$ be in $\mathcal{G}$.
Along $L$ the Tutte polynomial of $G$ has the form
\[T(G;1,y) = \sum_{\substack{A \subseteq E(G): \\ \rho(A)=\rho(G)}}(y-1)^{|A|-\rho(G)} = \sum_{i = -\rho(G)}^{|E(G)|}t_i y^i,\]
for some $t_{-\rho(G)}, \ldots , t_{|E(G)|}$.

The proof now follows in a similar way to that of Proposition~\ref{main proposition 1} by computing
$T(G^k;1,b)$ for $k=1, \ldots, |E(G)|+\rho(G)+1$ and then determining each coefficient $t_i$ in time polynomial in $|E(G)|$.
\end{proof}

The following proposition deals with the case when $b=1$.

\begin{proposition}
Let $L=H_0^y$ and $a \in \mathbb{Q} -\{1\}$. Then
\[ \pi_2[\mathcal{G},L] \propto_T \pi_3[\mathcal{G},a,1].\]
\end{proposition}

\begin{proof}
Let $G$ be in $\mathcal{G}$.
Along $L$ the Tutte polynomial of $G$ has the form
\[T(G;x,1) = \sum_{\substack{A \subseteq E(G):\\ \rho(A)=|A|}}(x-1)^{\rho(G)-\rho(A)}= \sum_{i =0}^{\rho(G)}t_i x^i,\]
for some $t_{0}, \ldots , t_{\rho(G)}$.

We now show that we can determine all of the coefficients $t_i$ from the evaluations $T(G^k;a,1)$ for $k= 1,\ldots, \rho(G)+1$
in time polynomial in $|E(G)|$. For each such $k$, $G^k$ may be constructed from $G$ in time polynomial in $|E(G)|$ and is bipartite, planar and connected.
By Theorem~\ref{greedoid thickening}, we have
\[T(G^k;a,1) = k^{\rho(G)}T\left(G; \frac{a+k-1}{k},1\right).\]

Therefore we may compute $T\left(G; \frac{a+k-1}{k},1\right)$ from $T(G^k;a,1)$.
Clearly $\left(\frac{a+k-1}{k},1\right)$ lies on $H_0^y$.
Since $a \neq 1$, the points $\left(\frac{a+k-1}{k},1\right)$ are pairwise distinct for $k = 1,2,\ldots, \rho(G)+1$.
Therefore by evaluating $T(G^k;a,1)$ for $k=1,\ldots, \rho(G)+1$, we obtain $\sum_{i = 0}^{|\rho(G)|}t_iz^i$ for $\rho(G)+1$ distinct values of $z$. This gives us $\rho(G)+1$ linear equations for the coefficients $t_i$. Again the matrix of the equations is a Vandermonde matrix and clearly non-singular. So, we may apply Lemma~\ref{lem:gauss} to compute $t_i$ for all $i$ in time polynomial in $|E(G)|$.
\end{proof}

We now summarize the three preceding propositions.

\begin{proposition}
\label{main subtheorem}
Let $L$ be either $H_0^x$, $H_0^y$, or $H_{\alpha}$ for $\alpha \in \mathbb{Q}-\{0\}$. Let $(a,b)$ be a point on $L$ such that $(a,b) \neq (1,1)$ and $b \notin \{-1,0\}$. Then
\[ \pi_2[\mathcal{G}, L] \propto_T \pi_3[\mathcal{G},a,b].\]
\end{proposition}

We now consider the exceptional case when $b=-1$. For reasons that will soon become apparent, we recall from Example~\ref{eg:littletrees} that $T(P_2;x,y) = x^2y-2xy+x+y$ and $T(S_k;x,y) = x^k$.

\begin{proposition}
\label{prop: b=-1}
Let $L$ be the line $y=-1$. For $a \notin \{\frac{1}{2},1\}$ we have
\[\pi_2[\mathcal{G},L] \propto_T \pi_3[\mathcal{G}, a,-1].\]
\end{proposition}

\begin{proof}
Let $G$ be in $\mathcal G$ and let $z=x-1$. Along $L$ the Tutte polynomial of $G$ has the form
\[T(G;x,-1) = \sum_{A \subseteq E(G)}z^{\rho(G)-\rho(A)}(-2)^{|A|-\rho(A)} = \sum^{\rho(G)}_{i=0} t_i z^i\]
for some $t_{0}, \ldots, t_{\rho(G)}$.

We now show that, apart from a few exceptional values of $a$, we can determine all of the coefficients $t_i$ in polynomial time from $T(G\sim S_k;a,-1)$, for $k= 0, 1, \ldots, \rho(G)$,  in time polynomial in $|E(G)|$. For each such $k$, $G\sim S_k$ may be constructed from $G$ in time polynomial in $|E(G)|$ and is bipartite, planar and connected.

By Theorem~\ref{attachment function greedoids} we have
\[T(G \sim S_k;a,-1) = a^{k \rho(G)}T\left(G;\frac{(a-1)^{k+1}(-1)^k}{a^k}+1,-1 \right).\]
Providing $a \neq 0$ we may compute $T\left(G;\frac{(a-1)^{k+1}(-1)^k}{a^k}+1,-1 \right)$ from $T(G\sim S_k;a,-1)$.
For $a \notin\{\frac{1}{2},1\}$ the points $\left(\frac{(a-1)^{k+1}(-1)^k}{a^k}+1,-1\right)$ are pairwise distinct for $k=0,1,2,\ldots, \rho(G)$.
Therefore by evaluating $T(G\sim S_k;a,-1)$ for $k=0,1,2,\ldots, \rho(G)$ where $a \notin \{0,\frac{1}{2},1\}$, we obtain $\sum_{i = 0}^{\rho(G)}t_iz^i$ for $\rho(G)+1$ distinct values of $z$.
This gives us $\rho(G)+1$ linear equations for the coefficients $t_i$. Again the matrix corresponding to these equations is a Vandermonde matrix
and clearly non-singular. So, we may apply Lemma~\ref{lem:gauss} to compute $t_i$ for all $i$ in time polynomial in $|E(G)|$.
Hence for $a \notin \{0,\frac{1}{2},1\}$, $\pi_2[\mathcal G,L] \propto \pi_3[\mathcal G,a,-1]$.

We now look at the case when $a=0$.
Note that $T(P_2;0,-1) = -1$.
Applying Theorem~\ref{attachment function greedoids} to $G$ and $P_2$ gives
\[
T(G\sim P_2;0,-1) = (-1)^{\rho(G)}T\left(G; \frac{(-1)^3(-1)^2}{-1}+1,-1\right)
=(-1)^{\rho(G)}T(G;2,-1).
\]
Therefore we have the reductions
\[\pi_2[\mathcal{G},L] \propto_T \pi_3[\mathcal{G},2,-1] \propto_T \pi_3[\mathcal{G},0,-1].\]
Since the Turing reduction relation is transitive, this implies that evaluating the Tutte polynomial at the point $(0,-1)$ is at least as hard as evaluating it along the line $y=-1$. This completes the proof.
\end{proof}

We now begin to classify the complexity of $\pi_3$. The next results will establish hardness for a few special cases, namely when $b \in \{-1,0,1\}$.

\begin{proposition}
\label{proposition (1,b)}
The problem $\pi_3[\mathcal{G},1,b]$ is $\#$P-hard apart from when $b=1$, in which case it has a polynomial time algorithm.
\end{proposition}
\begin{proof}
The hardness part follows directly from Theorem~\ref{VW} and Proposition~\ref{prop:x=1}. We have already noted the existence of a polynomial time algorithm to solve $\pi_3[\mathcal{G},1,1]$.
\end{proof}

\begin{proposition}
\label{proposition (a,-1)}
The problem $\pi_3[\mathcal{G},a,-1]$ is $\#$P-hard apart from when $a = 1/2$, in which case it has a polynomial time algorithm.
\end{proposition}

\begin{proof}
First note that there is a polynomial time algorithm for $\pi_3[\mathcal{G},a,-1]$ because
$(\frac{1}{2},-1)$ lies on $H_1$.
Now let $L$ be the line $y=-1$. By Proposition~\ref{prop: b=-1} we have
\[ \pi_2[\mathcal{G},L] \propto_T \pi_3[\mathcal{G},a,-1]\]
for $a \notin \{\frac{1}{2},1\}$.
So
\[ \pi_3[\mathcal{G},1,-1] \propto_T \pi_3[\mathcal{G},a,-1]\]
for $a \neq 1/2$.
By Proposition~\ref{proposition (1,b)} we know that $\pi_3[\mathcal{G},1,-1]$ is $\#$P-hard. So the result follows.
\end{proof}

\begin{proposition}
\label{prop: b=0}
The problem $\pi_3[\mathcal{G},a,0]$ is $\#$P-hard apart from when $a = 0$, in which case it has a polynomial time algorithm.
\end{proposition}
\begin{proof}
Let $G$ be in $\mathcal{G}$.
First note that evaluating the Tutte polynomial of $G$ at the point $(0,0)$ is easy since $(0,0)$ lies on the hyperbola $H_1$.

The rooted graph $G\sim S_1$ may be constructed from $G$ in time polynomial in $|E(G)|$ and is bipartite, planar and connected.
Applying Theorem~\ref{attachment function greedoids} to $G$ and $S_1$ gives
\[T(G\sim S_1; a, 0) = a^{\rho(G)}T(G;1,0).\]
Since $a \neq 0$ we may compute $T(G;1,0)$ from $T(G \sim S_1;a,0)$. Therefore $\pi_3[\mathcal{G},1,0] \propto \pi_3[\mathcal{G},a,0]$.
By Proposition~\ref{proposition (1,b)}, $\pi_3[\mathcal{G},1,0]$ is $\#$P-hard, and the result follows.
\end{proof}

Recall from Equation~\ref{characteristic polynomial specialization of tutte} that along $y=0$ the Tutte polynomial of a rooted graph specializes to the characteristic polynomial. Therefore we have the following corollary.
\begin{corollary}
Computing the characteristic polynomial $p(G;k)$ of a connected rooted graph $G$ is $\#$P-hard for all $k \in \mathbb{Q}-\{1\}$. When $k=1$, there is a polynomial time algorithm.
\end{corollary}
\begin{proof}
Let $k$ be in $\mathbb{Q}$. We have
\[p(G;k) = (-1)^{\rho(G)}T(G;1-k,0).\] By Proposition~\ref{prop: b=0} evaluating $T(G;1-k,0)$ is $\#$P-hard providing $k \neq 1$. Furthermore when $k=1$ we have
\[p(G;1) = (-1)^{\rho(G)}T(G;0,0) = \left\{ \begin{array}{ll} 1 & \text{ if $G$ is edgeless; }\\ 0 & \text{ otherwise, }\\
\end{array}\right.\]
and so it is easy to compute (as expected since $(0,0)$ lies on $H_1$).
\end{proof}

We now consider points along the line $y=1$.

\begin{proposition}
\label{proposition (a,1)}
The problem $\pi_2[\mathcal{G}, a,1]$ is $\#$P-hard when $a \neq 1$.
\end{proposition}

\begin{proof}
Let $G$ be a connected, planar, bipartite, unrooted graph with $V(G)=\{v_1, \ldots, v_n\}$.
Now for $1\leq j\leq n$, let $G_j$ be the graph in $\mathcal G$ obtained from $G$ by choosing $v_j$ to be the root. Let $\rho_j$ denote the rank function of $G_j$ and $a_i(G_j)$ be the number of subsets $A$ of the edges of $G_j$ having size $i$ so that the root component of $G|A$ is a tree.
Then
\[
T(G_j;x,1) = \sum_{\substack{A \subseteq E: \\ \rho_j(A)=|A|}}(x-1)^{\rho(G_j)-|A|} =\sum_{i=0}^{\rho(G_j)}a_i(G_j)(x-1)^{\rho(G_j)-i}.
\]
Let $a_i(G)$ denote the number of subtrees of $G$ with $i$ edges. Then
\[a_i(G) = \sum_{j=1}^n \frac{a_i(G_j)}{i+1}.\] This is because every subtree $T$ of $G$ with $i$ edges has $i+1$ vertices and its edge set is one of the sets $A$ contributing to $a_i(G_j)$ for the $i+1$ choices of $j$ corresponding to its vertices.

Given an oracle for $\pi_2[\mathcal{G},H_0^y]$, we can compute $a_i(G_j)$ for $i=0,\ldots,|E(G)|$ and $1\leq j \leq n$ in time polynomial in $|E(G)|$. So we can compute $a_i(G)$ and consequently the number of trees of $G$ in time polynomial in $|E(G)|$. Thus
\[ \#\text{SUBTREES} \propto_T \pi_3[\mathcal{G},H_0^y].\]
By Proposition~\ref{main subtheorem} we have
\[\#\text{SUBTREES} \propto_T\pi_3[\mathcal{G},H_0^y] \propto_T \pi_2[\mathcal{G},a,1]\]
for $a\ne 1$. The result now follows from Proposition~\ref{prop:jerrum}.
\end{proof}

We now summarize our results and prove Theorem~\ref{maintheoremrootedgraph}.
\begin{proof}[Proof of Theorem~\ref{maintheoremrootedgraph}]
Let $(a,b)$ be a point on $H_{\alpha}$ for some $\alpha$ in $\mathbb{Q} - \{0,1\}$.
By Proposition~\ref{main subtheorem} we have $\pi_2[\mathcal{G}, H_{\alpha}] \propto_T \pi_3[\mathcal{G},a,b]$ providing $b \notin \{-1,0\}$.
The hyperbola $H_{\alpha}$ crosses the $x$-axis at the point $(1-\alpha, 0)$.
By Proposition~\ref{prop: b=0} the problem $\pi_3[\mathcal{G},1-\alpha,0]$ is $\#$P-hard since $\alpha \neq 1$.
This gives us a $\#$P-hard point on each of these curves and therefore implies $\pi_2[\mathcal{G},H_{\alpha}]$ is $\#$P-hard for $\alpha \in \mathbb{Q}-\{0,1\}$.
Hence $\pi_3[\mathcal{G},a,b]$ is $\#$P-hard for $(a,b)\in H_{\alpha}$ with $\alpha \in \mathbb{Q}-\{0,1\}$ and $b \neq -1$.
The rest of the proof now follows directly by Propositions~\ref{proposition (1,b)},~\ref{proposition (a,-1)} and~\ref{proposition (a,1)}, and the discussion concerning the easy points at the beginning of the section.
\end{proof}

\section{Rooted Digraphs}
\label{sec:rooted digraph}
In this section we let $\mathbb G$ be the class of directed branching greedoids of root-connected rooted digraphs, a class we denote by $\mathcal D$. We consider the same three problems as in the previous section.
Again, it is more convenient to think of the input as being a root-connected rooted digraph rather than its directed branching greedoid.
We present analogous results to those in the previous section by finding the computational complexity of evaluating the Tutte polynomial of a root-connected digraph at a fixed rational point, eventually proving Theorem~\ref{maintheoremdigraph}.

We begin the proof by examining the easy points.
Let $D$ be a rooted digraph with edge set $E$ and rank function $\rho$.
If a point $(a,b)$ lies on the hyperbola $H_1$ then, following the remarks at the end of Section~\ref{The Tutte Polynomial of a Rooted Graph and of a Rooted Digraph}, $T(D;a,b)$ is easily computed.
We now show that evaluating $T(D;a,0)$ is easy for all $a \in \mathbb{Q}$.
A \emph{sink} in a digraph is a non-isolated vertex with no outgoing edges. Suppose that $D$ is a root-connected, rooted digraph with $s$ sinks. Then Gordon and McMahon~\cite{GordonMcMahon3} have shown that its characteristic polynomial $p$ satisfies the following.

\[p(D;\lambda) = \begin{cases} (-1)^{\rho(D)}(1-\lambda)^s & \text{if $D$ is acyclic;} \\
0 & \text{if $D$ has a directed cycle.} \end{cases} \]
Using the relation $T(D;1-\lambda,0) = (-1)^{\rho(D)}p(D;\lambda)$ we see that
\[T(D;x,0) = \begin{cases} x^s & \text{if $D$ is acyclic;} \\
0 & \text{if $D$ has a directed cycle.} \end{cases}\]
It is easy to count the sinks in a digraph so the problem $\pi_3[\mathcal{D},a,0]$ can be solved in polynomial time for any $a \in \mathbb{Q}$. Every edge of a component of a rooted digraph other than the root component is a greedoid loop, so if $D$ has such an edge then $T(D;1-\lambda,0)=0$. Furthermore, the addition or removal of isolated vertices makes no difference to $T(D)$.
So $T(D;a,0)$ can be computed in polynomial time for the class of all rooted digraphs.

We noted in Section~\ref{The Tutte Polynomial of a Rooted Graph and of a Rooted Digraph} that $T(D;1,1)$ is the number of spanning arborescences of the root component of $D$ rooted at $r$. This can be computed in polynomial time using the Matrix--Tree theorem for directed graphs~\cite{MR1579119,MR27521}.

We now move on to consider the hard points. The $k$-thickening operation will again be crucial: the \emph{$k$-thickening} $D^k$ of a root-connected digraph $D$ is obtained by replacing every edge $e$ in $D$ by $k$ parallel edges that have the same direction as $e$. We have $\Gamma(D^k)\cong (\Gamma(D))^k$, so
Theorem~\ref{greedoid thickening} can be applied to give an expression for $T(D^k)$.

The proof of the following proposition is omitted as it is analogous to that of Proposition~\ref{main subtheorem}.

\begin{proposition}
\label{digraph reduction 1}
Let $L$ be either $H_0^x, H_0^y$, or $H_{\alpha}$ for $\alpha \in \mathbb{Q}-\{0\}$. Let $(a,b)$ be a point on $L$ such that $(a,b) \neq (1,1)$ and $b \notin \{-1,0\}$. Then
\[\pi_2[\mathcal{D},L] \propto_T \pi_3[\mathcal{D},a,b].\]
\end{proposition}

We let $\overrightarrow{P_k}$ be the root-connected directed path of length $k$ with the root being one of the leaves and $\overrightarrow{S_k}$ be the root-connected directed star with $k$ edges emanating from the root.
Then $T(\overrightarrow{P_k};x,y) = 1+\sum_{i=1}^k(x-1)^iy^{i-1}$ and $T(\overrightarrow{S_k};x,y)=x^k$.
The proof of the following proposition is analogous to that of Proposition~\ref{prop: b=-1} with $\overrightarrow{P_k}$ and $\overrightarrow{S_k}$ playing the roles of $P_k$ and $S_k$.

\begin{proposition}
\label{digraph reduction 2}
Let $L$ be the line $y=-1$. For $a \notin \{\frac{1}{2},1\}$ we have \[\pi_2[\mathcal{D},L] \propto_T \pi_3[\mathcal{D},a,-1].\]
\end{proposition}

Next we classify the complexity of $\pi_3[\mathcal{D},1,b]$ for $b \notin \{0,1\}$.
Suppose we have a root-connected digraph $D$ and generate a random subgraph $(D,p)$ of $D$ by deleting each edge with probability $p$ independently of all the other edges. Let $g(D;p)$ denote the probability that $(D,p)$ is root-connected and let $g_j$ be the number of subsets $A$ of $E(D)$ with size $j$ so that $D|A$ is root-connected. Notice that $g_j$ is equal to the number of subsets $A$ of $E$ with $|A|=j$ and $\rho(A)=\rho(E)$. Then
\[ g(D;p) = \sum_{j=0}^{|E(D)|}g_j p^{|E(D)|-j}(1-p)^{j}.\]

Provan and Ball~\cite{ProvanBall} showed that the following problem is $\#$P-complete for each rational $p$ with $0 < p < 1$, and computable in polynomial time when $p=0$ or $p=1$.

\prob{$\#$\textsc{Connectedness Reliability}}{$D \in \mathcal{D}$.}{$g(D;p)$.}

Note that we have restricted the input digraph to being root-connected which Provan and Ball did not, but this does not make a difference to the complexity, because if $D$ is not root-connected then clearly $g(D;p)=0$.
We now use this result to classify the complexity of points along the line $x=1$.
\begin{proposition}
\label{digraph hard b>1}
The computational problem $\pi_3[\mathcal{D},1,b]$ is $\#$P-hard for $b>1$.
\end{proposition}
\begin{proof}
Let $D$ be a root-connected digraph with edge set $E$ and rank function $\rho$.
Then for $0<p<1$ we have
\begin{align*}
g(D;p)&= \sum_{\substack{ A \subseteq E(D): \\ \rho(A) = \rho(D)}}p^{|E(D)|-|A|}(1-p)^{|A|}= p^{|E(D)|-\rho(D)}(1-p)^{\rho(D)}\sum_{\substack{ A \subseteq E(D): \\ \rho(A)=\rho(D)}}\left(\frac{1-p}{p}\right)^{|A|-\rho(A)}\\
&=p^{|E(D)|-\rho(D)}(1-p)^{\rho(D)}T\left(D;1,\frac{1}{p}\right). \end{align*}

Evaluating $g(D;p)$ is therefore Turing-reducible to evaluating $T(D;1,\frac{1}{p})$ for $0<p<1$. Therefore, $\pi_3[\mathcal{D},1,b]$ is $\#$P-hard for $b>1$.
\end{proof}

In order to determine the complexity of the point $\pi_3[\mathcal{D},1,-1]$, we introduce a new operation on root-connected digraphs which we call the \emph{$k$-digon-stretch}.
We define a \emph{tailed $k$-digon} from $u$ to $v$ to be the digraph defined as follows. The vertex set is $\{w_0=u,w_1,\ldots,w_k,w_{k+1}=v\}$. There is an edge $w_0w_1$ and a directed cycle of length $2$ on $w_i$ and $w_{i+1}$ for each $i$ with $1\leq i \leq k$.
An example of a tailed $k$-digon is shown in Figure~\ref{fig:ktaildigon}. (The labelling of the edges will be needed later.)
For a root-connected digraph $D$, the $k$-digon-stretch of $D$ is constructed by replacing every directed edge $uv$ in $D$ by a tailed $k$-digon from $u$ to $v$. We denote the $k$-digon-stretch of $D$ by $D_k$.

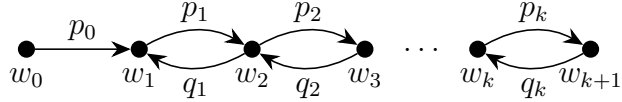
\begin{figure}
\centering
\begin{tikzpicture}
[>={Stealth[scale=1.3,angle'=45]},label distance=0pt,semithick,v/.style={circle,fill=black,draw=black,thick,
minimum size=2pt,inner sep=2pt
},e/.style={inner sep=4pt,rectangle,text height=1.5ex,text depth=.25ex,auto}]
\node(1)[v,label=below:$w_0$] at (0,0) {};
\node(2)[v,label=below:$w_1$] at (1.5,0) {};
\node(3)[v,label=below:$w_2$] at (3,0) {};
\node(4)[v,label=below:$w_3$] at (4.5,0) {};
\node(5)[v,label=below:$w_k$] at (6,0) {};
\node(6)[v,label=below:$w_{k+1}$] at (7.5,0) {};
\node(7) at (5.25,0) {$\ldots$};
\draw[->] (1) edge node [e] {$p_0$} (2);
\draw[->] (2) edge [bend left=30] node [e] {$p_1$} (3);
\draw[->] (3) edge [bend left=30] node [e,inner sep=1pt] {$q_1$} (2);
\draw[->] (3) edge [bend left=30] node [e] {$p_2$} (4);
\draw[->] (4) edge [bend left=30] node [e,inner sep=1pt] {$q_2$} (3);
\draw[->] (5) edge [bend left=30] node [e] {$p_k$} (6);
\draw[->] (6) edge [bend left=30] node [e,inner sep=1pt] {$q_k$} (5);
\end{tikzpicture}
\caption{A tailed $k$-digon.\label{fig:ktaildigon}}
\end{figure}

\begin{theorem}
\label{digon-stretch}
Let $D$ be a root-connected digraph. Then
\[T(D_k;1,y)=(k+1)^{|E(D)|-\rho(D)}y^{k |E(D)|}T\left(D;1,\frac{k+y}{k+1}\right).\]
\end{theorem}
\begin{proof}
Let $S$ be a subset of edges of a tailed $k$-digon from $u$ to $v$. If $S$ contains all the edges on the unique directed $uv$-path through the $k$-tailed digon, then $S$ is said to \emph{admit a $uv$-dipath}.
Let $A$ be a subset of $E(D_k)$ and $P(A)$ be the set of edges $uv$ in $D$ for which $A$ admits a $uv$-dipath.

We have $\rho(A) = \rho(D_k)$ if and only if
\begin{enumerate}[label=(\roman*),align=left,leftmargin=*]
\item for each directed edge $uv$ of $D$ and each vertex $w$ of the corresponding tailed $k$-digon from $u$ to $v$ in $D_k$, $A$ includes the edges of a path in the $k$-tailed digon from either $u$ or $v$ to $w$, and
\item $\rho(P(A))=\rho(D)$.
\end{enumerate}
Note that $\rho(D_k)= k|E(D)| + \rho(D)$.
We can write $A$ as the disjoint union $A = \bigcup_{e \in E(D)}A_e$ where $A_e$ is the set of edges of $A$ belonging to the tailed $k$-digon corresponding to $e$.
The Tutte polynomial of $D_k$ along the line $x=1$ is given by
\begin{align}
&T(D_k;1,y)=\sum_{\substack{ A \subseteq E(D_k):\\ \rho(A)=\rho(D_k)}}(y-1)^{|A|-\rho(D_k)} = \sum_{\substack{ B \subseteq E(D):\\ \rho(B)=\rho(D)}}\sum_{\substack{ A \subseteq E(D_k):\\ \rho(A)=\rho(D_k)\\P(A)=B}}(y-1)^{|A|-\rho(D_k)} \notag \\
&=\sum_{\substack{ B \subseteq E(D):\\ \rho(B)=\rho(D)}}\sum_{\substack{ A \subseteq E(D_k):\\ \rho(A)=\rho(D_k),\\ P(A)=B}}\underbrace{\left(\prod_{\substack{e \in E(D):\\ e \notin P(A)}}(y-1)^{|A_e|-k}\right)}_{(1)}\underbrace{\left(\prod_{\substack{e \in E(D):\\ e \in P(A)}}(y-1)^{|A_e|-(k+1)}\right)}_{(2)}(y-1)^{|P(A)|- \rho(D)}.
\label{end line in digon formula}\end{align}
Consider a tailed $k$-digon from $u$ to $v$ with vertex set labelled as described just before the statement of the theorem. For $0\leq i \leq k$, let $p_i$ denote the edge $w_iw_{i+1}$;
for $1\leq i \leq k$, let $q_i$ denote the edge $w_{i+1}w_{i}$.

In the first product above we are considering edges $e=uv$ for which $e\notin P(A)$. Thus $A_e$ does not contain all of $p_0$, \ldots, $p_k$. Let $j$ be the smallest integer such that $p_j\notin A_e$. As we are only interested in sets $A$ with $\rho(A)=\rho(D_k)$, each of $q_{j+1}$, \ldots, $q_k$ belongs to $A_e$. Thus $|A_e|\geq k$. Moreover each of $p_{j+1}$, \ldots, $p_k$ and $q_1$, \ldots, $q_j$ may or may not belong to $A_e$. As there are $k+1$ possibilities for $j$, summing \[\prod_{\substack{e \in E(D):\\ e \notin P(A)}}(y-1)^{|A_e|-k}\] over all possible choices of $A_e$ for $e\notin P(A)$ gives $\left((k+1)y^k\right)^{|E(D)|-|P(A)|}$.

In the second product above we are considering edges $e=uv$ for which $e\in P(A)$. Thus $A_e$ contains all of $p_0$, \ldots, $p_k$.
So $|A_e|\geq k+1$. Moreover each of $q_1$, \ldots, $q_k$ may or may not belong to $A_e$. Summing \[\prod_{\substack{e \in E(D):\\ e \in P(A)}}(y-1)^{|A_e|-(k+1)}\] over all possible choices of $A_e$ for $e\in P(A)$ gives $y^{k|P(A)|}$.

Thus the right side of Equation~\ref{end line in digon formula} becomes
\begin{align*}
\MoveEqLeft{\sum_{\substack{ B \subseteq E(D):\\ \rho(B)=\rho(D)}}y^{k|B|}\left((k+1)y^k\right)^{|{E(D)}|-|B|}(y-1)^{|B|- \rho(D)}}\\
&=y^{k|E(D)|}\sum_{\substack{ B \subseteq E(D):\\ \rho(B)=\rho(D)}}(k+1)^{\rho(B)-|B|+|E(D)|-\rho(D)}(y-1)^{|B|-\rho(B)}\\
&=y^{k|E(D)|}(k+1)^{|E(D)|-\rho(D)}T\left(D;1,\frac{y+k}{k+1}\right).
\end{align*}
\end{proof}

We now complete the classification of complexity for points on the line $H_0^x$.

\begin{proposition}
\label{digraph hard x=1}
The problem $\pi_3[\mathcal{D},1,b]$ is $\#$P-hard for $b\notin \{0,1\}$.
\end{proposition}
\begin{proof}
For $b \notin \{-1,0,1\}$ the result follows immediately from Propositions~\ref{digraph reduction 1} and \ref{digraph hard b>1}.
By Theorem~\ref{digon-stretch}, if $D$ is root-connected, then
\[T(D_2;1,-1)= 3^{|E(D)|-\rho(D)}T\left(D;1,\frac{1}{3}\right).\]
As $D_2$ is root-connected and can be constructed from $D$ in polynomial time, $\pi_3(\mathcal D,1,\frac 13) \propto \pi_3(\mathcal D,1,-1)$,
so $\pi_3(\mathcal D,1,-1)$ is $\#$P-hard.
\end{proof}

We now show that evaluating the Tutte polynomial of a root-connected digraph at most points on the hyperbola $H_{\alpha}$ for $\alpha \neq 0$ is at least as hard as evaluating it at the point $(1+\alpha,2)$.

\begin{proposition}
\label{digraph reduction to y=2}
Let $\alpha$ be in $\mathbb{Q}-\{0\}$ and
$(a,b)$ be a point on  $H_{\alpha}$ with $b \notin \{-1,0\}$, then
\[\pi_3[\mathcal{D},1+\alpha,2] \propto_T \pi_3[\mathcal{D},a,b].\]
\end{proposition}
\begin{proof}For $\alpha$ in $\mathbb{Q} -\{0\}$,
the hyperbola $H_{\alpha}$ crosses the line $y=2$ at the point $(1+\alpha, 2)$.
By Proposition~\ref{digraph reduction 1}, we know that for any point $(a,b)$ on $H_{\alpha}$ with $b \notin \{-1,0\}$ we have
$\pi_3[\mathcal{D},1+\alpha,2]\propto_T \pi_2[\mathcal{D},H_{\alpha}] \propto_T \pi_3[\mathcal{D},a,b]$.
\end{proof}

We will now show that evaluating the Tutte polynomial of a root-connected digraph at most of the points on the line $y=2$ is $\#$P-hard. This will enable us to classify the complexity of most points lying on the hyperbola $H_{\alpha}$ for all $\alpha \in \mathbb{Q}-\{0\}$.

\begin{proposition}
\label{digraph hard along y=2}
The problem $\pi_3[\mathcal{D},a,2]$ is $\#$P-hard for $a \neq 2$.
\end{proposition}
\begin{proof}
We begin by proving that when $L$ is the line $y=2$ we have
\[ \pi_2[\mathcal{D},L] \propto_T \pi_3[\mathcal{D},a,2]\]
for $a \notin \{1,2\}$.
Let $D$ be a root-connected digraph and let $z = x-1$.
Along $L$ the Tutte polynomial of $D$ has the form
\[T(D;x,2) = \sum_{A \subseteq E(D)}z^{\rho(D)-\rho(A)} = \sum_{i=0}^{\rho(D)}t_iz^i\]
for some $t_0,t_1, \ldots, t_{\rho(D)}$.
We will now show that for most values of $a$, we may determine all of the coefficients $t_i$ in polynomial time from $T(D\sim \overrightarrow{S_k};a,2)$ for $k=0,1, \ldots, \rho(D)$. For each such $k$, $D\sim \overrightarrow{S_k}$ is root-connected and can be constructed in polynomial time.
By Theorem~\ref{attachment function greedoids}, we have
\[T(D \sim \overrightarrow{S_k};a,2) = a^{k \rho(D)}T\left(D; \frac{2^k(a-1)^{k+1}}{a^k}+1,2\right).\]
Therefore we may compute $T\left(D; \frac{2^k(a-1)^{k+1}}{a^k}+1,2\right)$ from $T(D \sim \overrightarrow{S_k};a,2)$ when $a \neq 0$.
For $a \notin \{0,\frac{2}{3},1,2\}$ the values of $\left(\frac{2^k(a-1)^{k+1}}{a^k}+1,2\right)$ are pairwise distinct for $k=0,1, \ldots, \rho(D)$. Therefore by evaluating $T(D \sim \overrightarrow{S_k};a,2)$ for $k=0, 1, \ldots, \rho(D)$ where $a \notin \{0,\frac{2}{3},1,2\}$, we obtain $\sum_{i=0}^{\rho(D)}t_iz^i$ for $\rho(D)+1$ distinct values of $z$.
This gives us $\rho(D)+1$ linear equations for the coefficients $t_i$, and so by Lemma~\ref{lem:gauss},
they may be recovered in polynomial time.
Hence evaluating the Tutte polynomial of a root-connected digraph along the line $y=2$ is Turing-reducible to evaluating it at the point $(a,2)$ for $a \notin \{0,\frac{2}{3},1,2\}$.

We now consider the cases where $a=0$ or $a=\frac{2}{3}$. The digraph $D\sim \overrightarrow{P_2}$ is root-connected and may be constructed in polynomial time.
By Theorem~\ref{attachment function greedoids}, we have
\[
T(D\sim \overrightarrow{P_2};0,2) = 2^{\rho(D)}T\left(D;\frac{(-1)^32^2}{2}+1,2\right)\\
= 2^{\rho(D)}T(D;-1,2).\]
Therefore $\pi_3[\mathcal{D},-1,2] \propto_T \pi_3[\mathcal{D},0,2]$.
Similarly we have
\[
T\left(D\sim \overrightarrow{P_2};\frac{2}{3},2\right) = 2^{\rho(D)}T\left(D;\frac{(-\frac{1}{3})^32^2}{2}+1,2\right)= 2^{\rho(D)}T\left(D;\frac{25}{27},2\right).\]
Therefore $\pi_3[\mathcal{D},25/27,2] \propto_T \pi_3[\mathcal{D},2/3,2]$.
Putting all this together we get $\pi_2[\mathcal{D},L] \propto_T \pi_3[\mathcal{D},a,2]$
for all $a$ in $\mathbb Q-\{1,2\}$.
Consequently $\pi_3[\mathcal{D},1,2] \propto_T \pi_3[\mathcal{D},a,2]$, for all $a$ in $\mathbb Q-\{2\}$.

By Proposition~\ref{digraph hard x=1}, we know that $\pi_3[\mathcal{D},1,2]$ is $\#$P-hard.
This completes the proof.
\end{proof}

\begin{theorem}
Let $\alpha$ be in $\mathbb{Q}-\{0,1\}$ and $(a,b)$ be a point on $H_{\alpha}$ with $b \neq 0$. Then $\pi_3[\mathcal{D},a,b]$ is $\#$P-hard.
\end{theorem}
\begin{proof}
Suppose first that $b\ne -1$.
By Proposition~\ref{digraph reduction to y=2}, $\pi_3[\mathcal{D},1+\alpha,2] \propto_T \pi_3[\mathcal{D},a,b]$.
As $\alpha \ne 1$, Proposition~\ref{digraph hard along y=2}, implies $\pi_3[\mathcal{D},a,b]$ is $\#$P-hard.

Now suppose that $b=-1$. As $(a,b) \notin H_1$, we have $a\ne \frac 12$.
So by Proposition~\ref{digraph reduction 2}, $\pi_3[\mathcal{D},1,-1] \propto_T \pi_3[\mathcal{D},a,-1]$.
By Proposition~\ref{digraph hard x=1}, $\pi_3[\mathcal{D},1,-1]$ is $\#$P-hard.
Therefore $\pi_3[\mathcal{D},a,-1]$ is $\#$P-hard.
\end{proof}

The only remaining points we need to classify are those lying on the line $y=1$.
To do this we prove that the problem of evaluating the Tutte polynomial of a root-connected digraph at most fixed points along this line is at least as hard as the analogous problem for rooted graphs.

\begin{theorem}
The problem $\pi_3[\mathcal{D},a,1]$ is $\#$P-hard for $a$ in $\mathbb{Q}-\{1\}$.
\end{theorem}

\begin{proof}
Let $G$ be a connected rooted graph with root $r$. Construct a rooted graph $D$ with root $r$ by replacing every edge of $G$ by a pair of oppositely directed edges. Then $D$ is root-connected and can be constructed from $G$ in polynomial time.
We can define a natural map $f:2^{E(D)} \rightarrow 2^{E(G)}$ so that $f(A)$ is the set of edges of $G$ for which at least one corresponding directed edge is included in $A$.

If $\rho_{G}(A)=|A|$ then the root component of $G|A$ is a tree and includes all the edges of $A$.
Similarly if $\rho_D(A')=|A'|$ then the root component of $D|A'$ is an arborescence rooted at $r$ and includes all the edges of $A'$.
For every subset $A$ of $E$ with $\rho_G(A) = |A|$, there is precisely one choice of $A'$ with $\rho_D(A')=|A'|$ and $f(A')=A$, obtained by directing all the edges of $A$ away from $r$.
Thus there is a one-to-one correspondence between subsets $A$ of $E$ with $\rho_G(A) = |A|$ and subsets $A'$ of $E(D)$ with $\rho_D(A')=|A'|$, and this correspondence preserves the sizes of the sets. Therefore we have
\[ T(D;x,1) = \sum_{\substack{A'\subseteq E(D):\\ |A'|=\rho_D(A')}}(x-1)^{\rho(D)-|A'|} = \sum_{\substack{A \subseteq E:\\ |A|=\rho_G(A)}}(x-1)^{\rho(G)-|A|}= T(G;x,1).\]
So $\pi_3[\mathcal{G},a,1] \propto_T \pi_3[\mathcal{D},a,1]$. So by Proposition~\ref{proposition (a,1)}, we deduce that $\pi_3[\mathcal{D},a,1]$ is $\#$P-hard for $a \neq 1$.
\end{proof}

\section{Binary Greedoids}\label{sec:binarygreedoids}
In our final section we let $\mathbb G$ be the class of binary greedoids.
We present analogous results to those in the previous section by finding the computational complexity of evaluating the Tutte polynomial of a binary greedoid at a fixed rational point, eventually proving Theorem~\ref{maintheorembinarygreedoid}. As before, it is convenient to think of the input as being a binary matrix rather than its binary greedoid.

We begin by examining the easy points of Theorem~\ref{maintheorembinarygreedoid}. Let $\Gamma$ be a binary greedoid with element set $E$ and rank function $\rho$.
If a point $(a,b)$ lies on the hyperbola $H_1$ then, following the remarks at the end of Section~\ref{The Tutte Polynomial of a Rooted Graph and of a Rooted Digraph} $T(\Gamma;a,b)$ is easily computed.

We now focus on the hard points. The $k$-thickening operation will again be crucial. Given a binary matrix $M$, the $k$-thickening $M^k$ of $M$ is obtained by replacing each column of $M$ by $k$ copies of the column. We have $\Gamma(M^k) = (\Gamma(M))^k$, so Theorem~\ref{greedoid thickening} can be applied to compute the $T(M^k)$ in terms of $T(M)$.
Let $I_k$ denote the $k\times k$ identity matrix. Then $\Gamma(I_k)\cong \Gamma(P_k)$, so $T(I_k)=T(P_k)=1 +\sum_{j=1}^k(x-1)^jy^{j-1}$.

The proof of the following proposition is analogous to that of Proposition~\ref{main subtheorem}, thus we omit it.
\begin{proposition}
\label{binary greedoid reduction 1}
Let $L$ be either $H_0^x, H_0^y$, or $H_{\alpha}$ for $\alpha \in \mathbb{Q}-\{0\}$. Let $(a,b)$ be a point on $L$ such that $(a,b) \neq (1,1)$ and $b \notin \{-1,0\}$. Then
\[\pi_2[\mathcal{B},L] \propto_T \pi_3[\mathcal{B},a,b].\]
\end{proposition}

A binary matroid is a matroid that can be represented over the finite field $\mathbb{Z}_2$. Every graphic matroid is also binary, so Theorem~\ref{VW} and Lemma~\ref{binary greedoid to matroid} imply that $\pi_2[\mathcal{B},1,b]$ is $\#$P-hard providing $b\neq 1$. This immediately gives the following.
\begin{proposition}
\label{binary greedoid hard along $x=1$}
The problem $\pi_3[\mathcal{B},1,b]$ is $\#$P-hard for all $b$ in $\mathbb Q-\{1\}$.
\end{proposition}

The following result has been announced by Vertigan in~\cite{zbMATH00750655} and slightly later in~\cite{zbMATH01464264}, but up until now no written proof has been published. For completeness, we provide a proof in Appendix~\ref{sec:appendix}.

\begin{theorem}[Vertigan]\label{thm:Vertigan}
Evaluating the Tutte polynomial of a binary matroid is $\#$P-hard at the point $(1,1)$.
\end{theorem}

Using this result, we are able to fill in the missing point $(1,1)$ from the previous result and also establish hardness along the line $y=1$.

\begin{proposition}\label{prop:onlybitneeded}
The problem $\pi_3[\mathcal{B},a,1]$ is $\#$P-hard for all $a$.
\end{proposition}
\begin{proof}
By Proposition~\ref{binary greedoid reduction 1} we have $\pi_2[\mathcal{B},H_0^y] \propto_T \pi_3[\mathcal{B},a,1]$ for $a \neq 1$. The result now follows Theorem~\ref{thm:Vertigan}.
\end{proof}

\begin{proposition}
Let $\Gamma$ be a binary greedoid and let $\Gamma'=\Gamma(I_k)$. Then
\[T(\Gamma\approx \Gamma';x,y) = T(\Gamma;x,y)(x-1)^ky^k + T(\Gamma;1,y)\Big(1+ \sum_{j=1}^k (x-1)^jy^{j-1}-(x-1)^ky^k\Big).\]
\end{proposition}
\begin{proof}
The proof follows immediately from Theorem~\ref{thm:fullrankattach}.
\end{proof}

We now classify the complexity of $\pi_3[\mathcal{B},a,b]$
when $b=0$ or $b=-1$.

\begin{proposition}
\label{binary greedoid hard along $y=0$}
The  problem $\pi_3[\mathcal{B},a,0]$ is $\#$P-hard for all $a \neq 0$.
\end{proposition}
\begin{proof}
Let $M$ be a binary matrix with linearly independent rows. Then from Theorem~\ref{thm:fullrankattach}, we have $T(M\approx I_1;a,0)=aT(M;1,0)$.
Therefore when $a \neq 0$ we have $\pi_2[\mathcal{B},1,0] \propto_T \pi_2[\mathcal{B},a,0]$.
The result now follows from Proposition~\ref{binary greedoid hard along $x=1$}.
\end{proof}

\begin{proposition}
The problem $\pi_3[\mathcal{B},a,-1]$ is $\#P$-hard for all $a\neq \frac 12$.
\end{proposition}

\begin{proof}
Let $M$ be a binary matrix with linearly independent rows.
We have
\[ (2a-1)T(M;1,-1) = T(M\approx I_1;a,-1) + (a-1)T(M;a,-1).\]
Thus, $\pi_3[\mathcal B,1,-1]\propto_T \pi_3[\mathcal B,a,-1]$. By using Proposition~\ref{binary greedoid hard along $x=1$}, we deduce that $\pi_{0}[\mathcal B,a,-1]$ is $\#$P-hard.
\end{proof}

Our final result completes the proof of Theorem~\ref{maintheorembinarygreedoid}.

\begin{theorem}
Let $(a,b)$ be a point in $H_{\alpha}$ for $\alpha \in \mathbb{Q}-\{0,1\}$ with $b\neq -1$. Then $\pi_3[\mathcal{B},a,b]$ is $\#$P-hard.
\end{theorem}
\begin{proof}
For $\alpha \in \mathbb{Q}-\{0,1\}$, the hyperbola $H_{\alpha}$  crosses the $x$-axis at the point $(1-\alpha,0)$. By Proposition~\ref{binary greedoid reduction 1} since $b\neq -1$ and $(a,b) \neq (1,1)$ we have
$\pi_3[\mathcal{B},1-\alpha,0] \propto_T \pi_3[\mathcal{B},a,b]$.
The result now follows from Proposition~\ref{binary greedoid hard along $y=0$}.
\end{proof}

\appendix
\section{Counting bases in a represented matroid}\label{sec:appendix}
In this appendix, we present a proof that counting the number of bases of a represented matroid is $\#$P-complete. More precisely,
we consider the following family of counting problems.
Let $\mathbb F$ be a field.

\prob{\textsc{Counting Bases of $\mathbb F$-Represented Matroids}}{A $(0,1)$-matrix $A$.}
{The number of bases of $M(A)$, the matroid represented by $A$ over the field $\mathbb F$.}

\begin{theorem}\label{thm:main}
For every field $\mathbb F$, \textsc{Counting Bases of $\mathbb F$-Represented Matroids} is $\#$P-complete.
\end{theorem}

A proof of this result was announced nearly 30 years ago by Dirk Vertigan --- it first seems to have been referred to in~\cite{zbMATH00750655} and slightly later in~\cite{zbMATH01464264}, where it is described as an unpublished manuscript --- but no written proof has been circulated. Sketches of the proof have been presented by Vertigan in talks, for example, at the Conference for James Oxley in 2019~\cite{Vertiganunpub}. The second author was present at this meeting and the material in this section has been produced from his incomplete recollection of the talk. All the key ideas are due to Vertigan but the details including any errors, omissions or unnecessary complications are due to the authors. As pointed out to us by Dillon Mayhew~\cite{Dillon}, Vertigan's proof presented in~\cite{Vertiganunpub} introduced an intermediate step involving weighted bases; our proof does not require this intermediate step but this comes at the cost of introducing a larger matrix in the reduction.
We provide the proof, partly as a service to the community because we know of several colleagues who have tried to recreate it, but primarily because a referee has pointed out the undesirability of relying on an unpublished result.
Although our original aim was only to establish the special case of Theorem~\ref{thm:main} relevant for our work, it turns out that little extra effort is required to prove Theorem~\ref{thm:main} in full generality.

We require very little matroid theory other than basic notions such as rank, circuits and the closure operator. As we work exclusively with matroids having representations drawn from a specific family of matrices considered over different fields, the claims we make about the associated matroids can easily be checked by considering the representing matrices. For background on matroids see~\cite{Oxley}.

A graph is \emph{simple} if it has no loops or parallel edges.
To prove hardness, we give a reduction from counting perfect matchings in a simple graph, a problem which is well-known to be $\#$P-complete~\cite{zbMATH03646278}. Clearly, it makes no difference to the complexity of counting perfect matchings if we forbid our graphs from having isolated vertices.
Given such a graph $G$ with $n$ vertices, we construct a family of matrices $\{A_i: 1 \leq i \leq \lfloor n/2\rfloor +1\}$ with entries in $\{0,1\}$.
By considering these matrices as being defined over different fields, we obtain two corresponding families of matroids. Which family arises depends on whether the field has characteristic two. Thus the proof of Theorem~\ref{thm:main} splits into two parts depending on whether the characteristic of the underlying field is two.

We shall generally think of matrices as coming with sets indexing their rows and columns. If $A$ is a matrix with sets $X$ and $Y$ indexing its rows and columns respectively, then we say that $A$ is an $X\times Y$ matrix.
For non-empty subsets $X'$ and $Y'$ of $X$ and $Y$, respectively, $A[X',Y']$ is the submatrix of $A$ obtained by deleting the rows indexed by elements of $X-X'$ and the columns indexed by elements of $Y-Y'$.

Suppose that $G$ is a simple graph without isolated vertices having vertex set $\{v_1,\ldots,v_n\}$ and edge set $\{e_1,\ldots,e_m\}$. Let $k$ be a strictly positive integer.
Let
\begin{align*} X&=\{v_1,\ldots,v_n, e_1,\ldots, e_m\} \cup \{ f_{i,j}: 1\leq i\leq m, 1\leq j \leq k\} \\
   \intertext{and} Y&=\{v_1,\ldots,v_n, e_1,\ldots, e_m\} \cup \{ w_{i,j}, x_{i,j}, y_{i,j}, z_{i,j} : 1\leq i\leq m, 1\leq j \leq k\}.\end{align*}
Here both $X$ and $Y$ include all the vertices and edges of $G$, together with several new elements.
The matrix $A_k$ is an $X\times Y$ matrix.
To specify its entries
suppose that $e_i$ has endvertices $v_a$ and $v_b$ with $a<b$.
Then for each $j$ with $1\leq j\leq k$, taking $X'=\{v_a,v_b,f_{i,j}\}$ and $Y'=\{v_a,v_b,e_i,w_{i,j},x_{i,j},y_{i,j},z_{i,j}\}$, we let
\[ {A_k}[X',Y'] =   \begin{bNiceMatrix}[first-row,first-col]
        & v_a & v_b & e_i & w_{i,j} & x_{i,j} & y_{i,j} & z_{i,j} \\
v_a     & 1   & 0   & 1   & 0       & 1       & 0       & 1 \\
v_b     & 0   & 1   & 1   & 0       & 0       & 1       & 1  \\
f_{i,j} & 0   & 0   & 0   & 1       & 1       & 1       & 1
\end{bNiceMatrix}.\]
We complete the definition of $A_k$ by setting every as yet unspecified entry to zero.

Fix $\mathbb F$ and let $N_k=M(A_k)$, that is, the matroid with element set $Y$ represented by $A_k$ considered over $\mathbb F$.
Taking $Y'$ as in the previous paragraph, if $\mathbb F$ has characteristic two, then $N_k|Y'$ is isomorphic to the Fano matroid $F_7$ and otherwise $N_k|Y'$ is isomorphic to the non-Fano matroid $F_7^-$ obtained from $F_7$ by relaxing the circuit-hyperplane $\{e_i,x_{i,j},y_{i,j}\}$.
Now let $M_k=N_k\setminus (V\cup E)$. Note that $r(M_k)=r(N_k)=|V|+|E|k$ and that for each vertex $v$ and edge $e$ of $G$, $N_k$ contains elements $e$ and $v$, but $M_k$ contains neither.

We shall show that for each $k$, every basis of $M_k$ corresponds to what we call a \emph{feasible template} of $G$, that is, a subgraph of $G$ in which some edges are directed (possibly in both directions) and some are labelled, satisfying certain properties which we describe below.
In particular, we will see that the bidirected edges in a feasible template form a matching in $G$. Furthermore, the number of bases of $M_k$ corresponding to each feasible template depends only on $k$ and the numbers of edges directed and labelled in each possible way, and is easily computed. By varying $k$ and counting the number of bases of $M_k$, we can recover the number of feasible templates with each possible number of bidirected edges. The number of feasible templates with $n/2$ bidirected edges is equal to the number of perfect matchings of $G$.

Let $G$ be a simple graph without isolated vertices, having vertex set $V=\{v_1,\ldots,v_n\}$ and edge set $E=\{e_1,\ldots, e_m\}$.
A \emph{template} of $G$ is a spanning subgraph of $G$ in which edges may be bidirected, that is, two arrows are affixed one pointing to each endvertex, (uni)directed or undirected, and are labelled according to the following rules.
\begin{itemize}
\item Every bidirected edge is unlabelled.
\item A (uni)directed edge $e=v_av_b$ with $a<b$ is labelled either $wx$ or $yz$ if $e$ is directed towards $a$ and is labelled either $wy$ or $xz$ if $e$ is directed towards $b$.
\item An undirected edge is labelled either $wz$ or $xy$.
\end{itemize}

Even though the matroid $M_k$ itself depends on whether $\mathbb F$ has characteristic two, the proofs of the two cases have a great deal in common. To prevent repetition we describe the common material here, before finishing the two cases separately.
For $1\leq i \leq m$ and $1 \leq j \leq k$, let $F_{i,j}= \{w_{i,j},x_{i,j},y_{i,j},z_{i,j}\}$ and for $1 \leq i \leq m$, let $F_i =\bigcup_{1 \leq j \leq k} F_{i,j}$. For all $i$ and $j$, the set $F_{i,j}$ is a circuit and $r(M_k\setminus F_{i,j})<r(M_k)$.
Let $B$ be a basis of $M_k$. Then
$1 \leq |B\cap F_{i,j}| \leq 3$. Moreover, for all $i$, $r(F_i)=k+2$ and $r(M_k\setminus F_i) \leq r(M_k)-k$, so $k \leq |B\cap F_i| \leq k+2$.
The main idea in the proof is to use templates to classify each basis $B$ of $M_k$ by specifying $|B\cap F_i|$ for each $i$ and when $|B\cap F_i|=k+1$, implying that $|B\cap F_{i,j}|=2$ for precisely one value $j^*$ of $j$, additionally specifying $B\cap F_{i,j^*}$.

Suppose edge $e_i$ joins vertices $v_a$ and $v_b$ in $G$ and $a<b$. If $|B\cap F_i| = k$, then $\cl_{N_k}(B\cap F_i)-E(M_k)=\emptyset$, and if
$|B\cap F_i| = k+2$, then $\cl_{N_k}(B\cap F_i)-E(M_k)=\{v_a,v_b,e_i\}$.
If $|B\cap F_i| = k+1$, then $|B \cap F_{i,j}| =2$ for precisely one value $j^*$ of $j$ and $\cl_{N_k}(B\cap F_i)-E(M_k)$ depends on $B\cap F_{i,j^*}$.
\begin{itemize}
\item If $B\cap F_{i,j^*}$ is $\{w,x\}$ or $\{y,z\}$, then $\cl_{N_k}(B\cap F_i)-E(M_k)=\{v_a\}$.
\item If $B\cap F_{i,j^*}$ is $\{w,y\}$ or $\{x,z\}$, then $\cl_{N_k}(B\cap F_i)-E(M_k)=\{v_b\}$.
\item If $B\cap F_{i,j^*}$ is $\{w,z\}$, then $\cl_{N_k}(B\cap F_i)-E(M_k)= \{e_i\}$.
\item If $B\cap F_{i,j^*}$ is $\{x,y\}$, then $\cl_{N_k}(B\cap F_i)-E(M_k)$ is $\{e_i\}$ when $\mathbb F$ has characteristic two and is empty otherwise.
\end{itemize}

To each subset $S$ of $E(M_k)$, such that for all $i$, $S\cap F_i$ is independent and for all $i$ and $j$, $|S\cap F_{i,j}|\geq 1$,
we associate a template $T(S)$ of $G$, by starting with an edgeless graph with vertex set $V(G)$ and doing the following for each edge $e_i$ of $G$ such that $|S\cap F_{i,j}|>1$ for some $j$. Suppose that $e_i=v_av_b$ with $a<b$. We first consider whether to add $e_i$ to $T(S)$ and whether to direct it.
\begin{itemize}
\item If $\cl_{N_k}(S\cap F_i)-E(M_k)=\{v_a,v_b,e_i\}$, then add $e_i$ to $T(S)$ and bidirect it.
\item If $\cl_{N_k}(S\cap F_i)-E(M_k)=\{v_a\}$, then add $e_i$ to $T(S)$ and direct it from $v_b$ to $v_a$.
\item If $\cl_{N_k}(S\cap F_i)-E(M_k)=\{v_b\}$, then add $e_i$ to $T(S)$ and direct it from $v_a$ to $v_b$.
\item If $\cl_{N_k}(S\cap F_i)-E(M_k)\subseteq \{e_i\}$, then add $e_i$ to $T(S)$ (and do not direct it).
\end{itemize}
In the last three cases above, we also label $e_i$. To do this let $j^*$ be the unique value of $j$ such that $|S\cap F_{i,j}|=2$.
Then label $e_i$ with the elements of $S\cap F_{i,j^*}$, but with their subscripts omitted. In this way the edge $e_i$ is given two labels from the set $\{w,x,y,z\}$.

\subsection{$\mathbb F$ has characteristic two}
We now focus on the case when $\mathbb F$ has characteristic two.
The following result is the key step in the proof.

\begin{proposition}\label{prop:keychartwo}
A subset $B$ of $E(M_k)$ is a basis of $M_k$ if and only if all of the following conditions hold.
\begin{enumerate}
\item For all $i$, $B\cap F_{i}$ is independent.
\item For all $i$ and $j$, $|B \cap F_{i,j}| \geq 1$.
\item The subgraph of $T(B)$ induced by its undirected edges is acyclic.
\item It is possible to direct the undirected edges of $T(B)$ so that every vertex has indegree one.
\end{enumerate}
\end{proposition}

\begin{proof}
We first show that the conditions are collectively sufficient. Suppose that $B$ satisfies each of the conditions and that $T(B)$ has $b$ bidirected edges, $r$ (uni)directed edges and $u$ undirected edges. Then the last condition implies that $2b+r+u=n$.
Combining this with the first two conditions gives $|B|= km + 2b + r + u = km+n=r(M_k)$. So, it is sufficient to prove that $r(B)=r(M_k)$. We will show that the last two conditions imply that $v_i \in \cl_{N_k}(B)$ for $i=1,\ldots,n$. Then the second condition ensures that $\cl_{N_k}(B)=E(N_k)$ and consequently $r(B)=r(N_k)=r(M_k)$ as required.

Consider a vertex $v$ of $G$. The last two conditions imply that there is a (possibly empty) path $P$ in $T(B)$ between $v$ and a vertex $v'$ having indegree one
and comprising only undirected edges. Suppose that the vertices of $P$ in order are $v_{j_1}=v', v_{j_2}, \ldots, v_{j_l}=v$ and that for $1\leq h \leq l-1$, the edge joining $v_{j_h}$ and $v_{j_{h+1}}$ in $P$ is $e_{i_h}$.
Then $\{v_{j_1},e_{i_1},\ldots,e_{i_{l-1}}\} \subseteq \cl_{N_k}(B)$. As $v_{j_h} \in \cl_{N_k} (\{v_{j_{h-1}},e_{i_{h-1}}\})$ for $h=2,\ldots,l$, we see that $v=v_{j_l}\in \cl_{N_k}(B)$, as required. Thus the conditions are sufficient.

To show that each condition is necessary we suppose that $B$ is a basis of $M_k$.
Clearly the first condition is necessary. We observed earlier that for all $i$ and $j$, $r(E(M_k)-F_{i,j})<r(M_k)$, so the second condition is also necessary. Suppose, without loss of generality, that edges $e_1,\ldots,e_l$ are undirected and form a circuit in $T(B)$. Then the corresponding elements $e_1, \ldots, e_l$ form a circuit in $N_k$. Because each of $e_1,\ldots,e_l$ is undirected, $e_i \in \cl_{N_k}(B\cap F_{i})$ for $i=1,\ldots,l$. Thus
\[ e_{l} \in \cl_{N_k} (\{e_{1},\ldots,e_{l-1}\}) \subseteq \cl_{N_k}\Bigg( \bigcup_{i=1}^{l-1} (B\cap F_{i})\Bigg).\]
So there is a circuit of $N_k$ contained in $\{e_l\} \cup (B\cap F_{l})$ and another contained in
$\{e_{l}\} \cup \bigcup_{i=1}^{l-1} (B\cap F_{i})$. Hence there is a circuit of $N_k$ and consequently of $M_k$ contained in
$\bigcup_{i=1}^{l} (B\cap F_{i})$, contradicting the fact that $B$ is a basis. Thus the third condition is necessary.

Finally, suppose that $T(B)$ has $b$ bidirected edges, $r$ (uni)directed edges and $u$ undirected edges. Then, as $km+2b+r+u=|B|=r(M_k)=km+n$, we have $2b+r+u=n$. Observe that if the undirected edges are assigned a direction, then the sum of the indegrees of the vertices will become $n$. Suppose that it is impossible to direct the undirected edges of $T(B)$ so that each vertex has indegree one. Then, before directing the undirected edges, there must either be a vertex $z$ with indegree at least two, or two vertices $x$ and $y$ both having indegree at least one and joined by a path $P$ of undirected edges.

In either case the aim is to establish a contradiction by showing that there is some vertex $v$ such that $B\cup \{v\}$ contains two distinct circuits in $N_k$. Then $B$ contains a circuit of $N_k$ and consequently of $M_k$.
In the former case there are distinct edges $e_i$ and $e_j$ directed towards (and possibly away from as well) $z$ in $T(B)$. So $z \in \cl_{N_k} (B \cap F_i) \cap \cl_{N_k} (B\cap F_j)$ implying that $(B\cap F_i) \cup \{z\}$ and $(B\cap F_j) \cup \{z\}$ both contain circuits of $N_k$ including $z$. But then $(B\cap F_i)\cup (B\cap F_j)= B\cap (F_i \cap F_j)$ contains a circuit of $N_k$ and consequently of $M_k$, contradicting the fact that $B$ is a basis of $M_k$. So we may assume that the latter case holds. Suppose that, without loss of generality, the vertices of $P$ in order are $v_1=x,v_2,\ldots,v_l=y$. Suppose, again without loss of generality, that for $i=2,\ldots,l$, the edge joining $v_{i-1}$ and $v_{i}$ in $P$ is $e_i$, that $e_1$ is directed towards $x=v_1$ in $T(B)$ and $e_{l+1}$ is directed towards $y=v_l$ in $T(B)$. Then $y \in \cl_{N_k} (B\cap F_{l+1})$ and $x \in \cl_{N_k} (B\cap F_{1})$. Furthermore, for each $i=2,\ldots,l$,
$e_i \in \cl_{N_k} (B\cap F_i)$, so $v_i \in  \cl_{N_k}\Big( \bigcup_{j=1}^{i} (B\cap F_{j}) \Big)$. In particular,
$y  \in \cl_{N_k}\Big( \bigcup_{j=1}^{l} (B\cap F_{j}) \Big)$.
So, there is a circuit of $N_k$ contained in $\{y\} \cup \{B\cap F_{l+1}\}$ and another contained in
$\{y\} \cup \bigcup_{j=1}^{l} (B\cap F_{j})$. Hence there is a circuit of $N_k$ and consequently of $M_k$ contained in
$\bigcup_{j=1}^{l+1} (B\cap F_{j})$, contradicting the fact that $B$ is a basis.
It follows that it possible to direct the undirected edges of each component of $T(B)$ so that every vertex has indegree one, establishing the necessity of the final condition.
\end{proof}

We say that a template $T$ is \emph{feasible} if it satisfies the last two conditions in the previous result, that is, if the subgraph induced by its undirected edges is acyclic and every vertex of the graph obtained from $T$ by contracting the undirected edges has indegree equal to one.

\begin{proposition}\label{prop:counting}
Let $G$ be a simple graph without isolated vertices and let $T$ be a feasible template of $G$ with $b$ bidirected edges. Then the number of bases of $M_k$ with template $T$ is
\[ 4^{km}\Big(\frac k4\Big)^n \Big(\frac 4k + 12\Big)^b.\]
\end{proposition}
\begin{proof}
If follows from the definition of feasibility that if a feasible template contains $b$ bidirected edges, then it has $n-2b$ edges which are
either (uni)directed or undirected. Furthermore $G$ has $m-n+b$ edges which are not in $T$.
Suppose that $B$ is a basis with template $T$.
We count the number of choices for $B$. Suppose that $e_i$ is an edge of $G$ which is not present in $T$. Then for $j=1,\ldots,k$, we have $|F_{i,j} \cap B|=1$, so there are $4^k$ choices for $B\cap F_i$. Now suppose that $e_i$ is either (uni)directed or undirected in $T$. Then for all but one choice of $j$ in $1,\ldots,k$, we have $|F_{i,j}\cap B|=1$ and for the remaining possibility for $j$, $|F_{i,j}\cap B|=2$, with the choice of elements of $F_{i,j}$ specified by the labelling of the edge $e_i$.
Thus there are $k\cdot 4^{k-1}$ choices for $B\cap F_i$. Finally suppose that $e_i$ is a bidirected edge. Then there are two subcases to consider. Either $|F_{i,j}\cap B|=3$ for one value of $j$ and $|F_{i,j}\cap B|=1$ for all other values of $j$, or $|F_{i,j}\cap B|=2$ for two values of $j$ and $|F_{i,j}\cap B|=1$ for all other values of $j$. Suppose that $|F_{i,j'}\cap B|=|F_{i,j''}\cap B|=2$ for $j'\ne j''$. Then we also require that
$\cl_{N_k}(B\cap F_{i,j'})-F_i \ne  \cl_{N_k}(B\cap F_{i,j'})-F_i$. Thus there are $k \cdot 4^k + \binom k2 \cdot 6 \cdot 4 \cdot 4^{k-2}$ choices for $B\cap F_i$.

So the number of bases of $M_k$ with template $T$ is
\[ (4^k)^{m-n+b} (k\cdot 4^{k-1})^{n-2b} (k \cdot 4^k + \binom k2 \cdot 6 \cdot 4 \cdot 4^{k-2})^b = 4^{km}\Big(\frac k4\Big)^n \Big(\frac 4k + 12\Big)^b.\]
\end{proof}

\begin{theorem}\label{thm:chartwo}
If $\mathbb F$ is a field with characteristic two, then the problem \textsc{Counting Bases of $\mathbb F$-Represented Matroids} is $\#$P-complete.
\end{theorem}

\begin{proof}
It is clear that \textsc{Counting Bases of $\mathbb F$-Represented Matroids} belongs to $\#$P. To prove hardness, we give a reduction from counting perfect matchings. Let $G$ be a simple graph with $n$ vertices and $m$ edges. We may assume that $G$ has no isolated vertices and $n$ is even. We can construct representations of the matroids $M_1, \ldots, M_{n/2+1}$ in time polynomial in $n$ and $m$.
For $k=1,\ldots, n/2+1$, let
 $b_k$ denote the number of bases of $M_k$ and for $j=0,\ldots,n/2$, let $t_j$ denote the number of feasible templates of $G$ with $j$ bidirected edges. Then for $k=1,\ldots, n/2+1$, by Proposition~\ref{prop:counting}, we have
 \[ b_k = \sum_{j=0}^{n/2}  4^{km}\Big(\frac k4\Big)^n \Big(\frac 4k + 12\Big)^j t_j.\]
 Given $b_1,\ldots, b_{n/2+1}$, we may recover $t_0,\ldots,t_{n/2}$ in time polynomial in $n$ and $m$. In particular, we may recover $t_{n/2}$. But  feasible templates with $n/2$ bidirected edges are in one-to-one correspondence with perfect matchings of $G$.
 As counting perfect matching is $\#$P-complete by~\cite{zbMATH03646278}, we deduce that when $\mathbb F$ has characteristic two, \textsc{Counting Bases of $\mathbb F$-Represented Matroids} is $\#$P-complete.
\end{proof}

\subsection{$\mathbb F$ does not have characteristic two}
When $\mathbb F$ does not have characteristic two, we can proceed in a similar way, but the proof is a little more complicated, as we need to consider more carefully circuits of undirected edges in a template.
We say that a circuit of a template comprising only undirected edges is \emph{good} if it has an odd number of edges labelled $wz$. The following lemma gives us the key property of circuits of undirected edges in the template of a basis.

\begin{lemma}\label{lem:keycharnot2}
Let $G$ be a simple graph without isolated vertices and let $M_k$ and $N_k$ be the associated matrices. Let $C$ be a circuit of $G$ and select a set $Z$ of $2|C|$ elements of $M_k$ as follows. For each $i$ such that $e_i$ is an edge of $C$, choose $j$ with $1\leq j \leq k$ and add either $w_{i,j}$ and $z_{i,j}$, or $x_{i,j}$ and $y_{i,j}$ to $Z$. To simplify notation we omit the second subscript and for each $i$ denote the elements added to $Z$ by either $w_i$ and $z_i$, or $x_i$ and $y_i$.
Then both of the following hold.
\begin{enumerate}
\item If $|\{i: \{w_i,z_i\}\subseteq Z\}|$ is odd then $Z$ is independent in $M_k$ (and $N_k$) and for each vertex $v$ of $C$, $v\in \cl_{N_k}(Z)$.
\item If $|\{i: \{w_i,z_i\}\subseteq Z\}|$ is even then $Z$ is a circuit in $M_k$ (and $N_k$).
\end{enumerate}
\end{lemma}

\begin{proof}
For an edge $e_i$ of $C$, we say that $e_i$ is a $wz$-edge if $\{w_i,z_i\}\subseteq Z$, and otherwise we say that it is an $xy$-edge.
We first prove that $Z$ is either independent or a circuit, depending on the parity of $|\{i: \{w_i,z_i\}\subseteq Z\}|$. Consider the submatrix $A$ of $A_k$ containing just the columns indexed by members of $Z$ and consider the coefficients of a non-trivial linear combination of these columns summing to zero. As each row of $A$ is either zero or contains two non-zero entries, both equal to one, we may assume that the non-zero coefficients are all $\pm 1$. Furthermore, for every $wz$-edge $e_i$, the coefficients of $w_i$ and $z_i$ must sum to zero, and similarly for every $xy$-edge $e_i$, the coefficients of $x_i$ and $y_i$ must sum to zero. Now consider two adjacent edges $e_i$ and $e_j$ in $C$, and let $v$ be their common endvertex. As the
row indexed by $v$ contains one non-zero entry in a column indexed by an element of $\{w_i,x_i,y_i,z_i\}\cap Z$ and also one in a column indexed by an element of $\{w_j,x_j,y_j,z_j\}\cap Z$, we deduce that the coefficients of $\{w_i,x_i,y_i,z_i\}\cap Z$ are non-zero if and only those of $\{w_j,x_j,y_j,z_j\}\cap Z$ are non-zero. Consequently all the coefficients in a non-trivial linear combination of the columns of $A$ are non-zero.
Now imagine traversing $C$ in $G$ and suppose that $e_i$ and $e_j$ are consecutive (not necessarily adjacent) $wz$-edges. Then it is not difficult to see that the coefficients of $w_i$ and $w_j$ (and of $z_i$ and $z_j$) have opposite signs. Thus, if
there is an odd number of $wz$-edges, then no non-trivial linear combination of the columns of $A$ sums to zero and $Z$ is independent. Alternatively, if
there is an even number of $wz$-edges, then one can assign coefficients $\pm 1$ to columns indexed by $w_i$ or $z_i$ meeting the necessary conditions we have established, and then it is not difficult to check that non-zero coefficients may be assigned to all the remaining columns in order to give a non-trivial linear combination of the columns of $A$ summing to zero. Thus $Z$ is dependent, and as we have shown that 
all coefficients of a non-trivial linear combination of the 
columns of $A$ summing to zero must be non-zero, we deduce that $Z$ is a circuit.

Finally, suppose that there is an odd number of $wz$-edges and let $V(C)$ denote the vertex set of $C$. Then $r_{N_k}(Z\cup V(C))=2|C|=r_{N_k}(Z)$, so for each vertex $v$ of $C$, $v\in \cl_{N_k}(Z)$.
\end{proof}

The analogue of Proposition~\ref{prop:keychartwo} is as follows.

\begin{proposition}
A subset $B$ of $E(M_k)$ is a basis of $M_k$ if and only if all of the following conditions hold.
\begin{itemize}
\item For all $i$, $B\cap F_{i}$ is independent.
\item For all $i$ and $j$, $|B \cap F_{i,j}| \geq 1$.
\item Every circuit of $T(B)$ comprising only undirected edges is good.
\item It is possible to direct the undirected edges of $T(B)$ so that every vertex has indegree one.
\end{itemize}
\end{proposition}

\begin{proof}
Most of the proof follows that of Proposition~\ref{prop:keychartwo}. The main difference concerns circuits of $T(B)$ comprising undirected edges.
To prove the sufficiency of the conditions we modify the last part of the sufficiency argument of Proposition~\ref{prop:keychartwo}.
Suppose that $B$ satisfies each of the conditions. The key step involves showing that for every vertex $v$ in $G$, we have $v\in\cl_{N_k}(B)$.
The last two conditions imply that for a vertex $v$ of $G$,
there is a (possibly empty) path $P$ in $T(B)$, comprising only undirected edges, between $v$ and a vertex $v'$, which either has indegree one
or belongs to a good circuit. Using Lemma~\ref{lem:keycharnot2} for the latter case, we see that in either case $v' \in \cl_{N_k}(B)$ and the proof may continue in the same way as that of Proposition~\ref{prop:keychartwo}.

To show that each condition is necessary we suppose that $B$ is a basis of $M_k$. The necessity of the first two conditions follows in the same way as in the proof of Proposition~\ref{prop:keychartwo} and the necessity of the third follows from Lemma~\ref{lem:keycharnot2}.
The necessity of the final condition follows from a similar argument to that used in the proof of Proposition~\ref{prop:keychartwo}, but there are more cases to consider. Notice that the necessity of the third condition implies that every undirected edge of $T(B)$ belongs to at most one circuit comprising only undirected edges. If it is not possible to direct the undirected edges of $T(B)$ so that each edge has indegree one, then before directing the undirected edges one of the following must occur.
\begin{enumerate}
\item There is a vertex $z$ of $T(B)$ with indegree at least two.
\item There is a vertex $z$ belonging to two edge-disjoint good circuits.
\item There is a vertex $z$ of $T(B)$ with indegree one which belongs to a good circuit.
\item There are vertices $x$ and $y$ of $T(B)$ not belonging to the same good circuit and joined by a path $P$ comprising undirected edges and so that each of $x$ and $y$ either has indegree one or belongs to a good circuit.
\end{enumerate}
To show that each possibility leads to a contradiction, the aim is again to show that there is a vertex $v$ of $B$ such that $B\cup \{v\}$ contains two distinct circuits of $N_k$. The first case is the same as in the proof of Proposition~\ref{prop:keychartwo}. The second and third follow similarly with the aid of Lemma~\ref{lem:keycharnot2} and the final one follows in a similar way to the analogous case in Proposition~\ref{prop:keychartwo}, noting first that by Lemma~\ref{lem:keycharnot2}, if necessary, there are disjoint subsets $B_x$ and $B_y$ of $B$ with $x\in \cl_{N_k}(B_x)$ and $y\in \cl_{N_k}(B_y)$ and then deducing that $y$ (and in fact every vertex of $P$) belongs to $\cl_{N_k}(B_x)$.
\end{proof}

We amend the definition of feasibility to say that a template $T$ is \emph{feasible} if it satisfies the last two conditions in the previous result, that is, if the subgraph induced by its undirected edges contains no circuits including an even number of edges labelled $wz$ and it is possible to direct the undirected edges of $T(B)$ so that every vertex has indegree one.

\begin{proposition}\label{prop:counting2}
Let $G$ be a simple graph without isolated vertices and let $T$ be a feasible template of $G$ with $b$ bidirected edges. Then the number of bases of $M_k$ with template $T$ is
\[ 4^{km}\Big(\frac k4\Big)^n \Big(\frac 3k + 13\Big)^b.\]
\end{proposition}
\begin{proof}
The proof is very similar to that of Proposition~\ref{prop:counting}. The key difference is counting the number of choices for $F_i$ when $i$ is a bidirected edge and $|F_{i,j'}\cap B|=|F_{i,j''}\cap B|=2$ for $j'\ne j''$.  There are now $26$ ways to choose $F_{i,j'}$ and $F_{i,j''}$ compared with $24$ when $\mathbb F$ has characteristic two.

So the number of bases of $M_k$ with template $T$ is
\[ (4^k)^{m-n+b} (k\cdot 4^{k-1})^{n-2b} (k \cdot 4^k + \binom k2 \cdot 26 \cdot 4^{k-2})^b = 4^{km}\Big(\frac k4\Big)^n \Big(\frac 3k + 13\Big)^b.\]
\end{proof}

\begin{theorem}
If $\mathbb F$ is a field with characteristic other than two, then the problem \textsc{Counting Bases of $\mathbb F$-Represented Matroids} is $\#$P-complete.
\end{theorem}

The proof is identical to that of Theorem~\ref{thm:chartwo}.

\section*{Acknowledgement}
We thank Mark Jerrum and Dillon Mayhew for helpful suggestions concerning Appendix~\ref{sec:appendix}.

\bibliography{difficultpoints}
\bibliographystyle{plain}

\end{document}